\documentclass{article}
\usepackage[left=3.5cm,right=3.5cm,top=3.5cm,bottom=3.5cm,bindingoffset=0cm]{geometry}

\usepackage{caption}

\usepackage[mathscr]{euscript}
\usepackage{amsmath} 
\usepackage{amsfonts}
\usepackage{amssymb}
\usepackage{ifthen} 
\usepackage{graphicx}
\usepackage{upgreek}
\usepackage{calligra}
\usepackage{calrsfs}
\usepackage{chngcntr}
\usepackage{cancel}
\usepackage{ifthen}
\usepackage{amsmath}
\usepackage{amsopn}
\usepackage{amsxtra}
\usepackage{amssymb}
\usepackage{amsbsy}
\usepackage{amscd}

\usepackage{xcolor}
\usepackage{color}

\usepackage{mathtools}
\usepackage{cases}
\usepackage[breaklinks]{hyperref}
\hypersetup{colorlinks, citecolor=blue, linkcolor=blue, backref=true}

\usepackage[sort&compress,square,comma,numbers]{natbib}

\usepackage{amsthm}
\newtheorem{theorem}{Theorem}[section]
\newtheorem{remark}[theorem]{Remark}
\newtheorem{lemma}[theorem]{Lemma}
\usepackage{mdframed}
\usepackage{graphicx}
\usepackage{mathtools}
\usepackage{cases}
\usepackage{mdframed}
\usepackage{graphicx}

\usepackage{upgreek}
\usepackage{epsfig}
\usepackage{shadethm}
\usepackage{amsthm}
\usepackage[mathscr]{euscript}
\usepackage{amsxtra}
\usepackage{amssymb}
\usepackage{calligra}
\usepackage{calrsfs}

\newcommand{\norm}[1]{\lVert#1\rVert}
\newcommand{\id}[1]{\operatorname{\mathcal{id}_{\mathnormal{#1}}}}
\DeclareMathAlphabet{\mathpzc}{OT1}{pzc}{m}{it}
\newcommand{\pd}[2]{\frac{\partial#1}{\partial#2}}

\newcommand{\beqans}{\begin{subequations}\begin{eqnarray}}
\newcommand{\eeqans}[1]{\end{eqnarray}\label{#1}\end{subequations}}

\usepackage{amsmath,environ}\NewEnviron{eqn}{\begin{equation}\begin{split} \BODY\end{split}\end{equation}}
\newcommand{\beqan}{\begin{eqnarray}}
\newcommand{\eeqan}{\end{eqnarray}}
\DeclareMathOperator{\sq}{\mathpzc{q}}

\begin{document}

\title{Continuum limit of discrete Sommerfeld problems on square lattice}
\author{Basant Lal Sharma\thanks{Department of Mechanical Engineering, Indian Institute of Technology Kanpur, Kanpur, U. P. 208016, India ({bls@iitk.ac.in}). Published in Sadhana, Volume 42(5), May 2017, Pages 713-728 DOI 10.1007/s12046-017-0636-6}}
\date{}
\maketitle
\begin{abstract}
{A low frequency approximation of the {\em discrete} Sommerfeld diffraction problems, involving the scattering of a time harmonic lattice wave incident
on square lattice {by} a discrete Dirichlet {or} a discrete Neumann half-plane,
is investigated. It is established that the exact solution of the discrete model converges to the solution of the continuum model, i.e. the {\em continuous} Sommerfeld problem, in certain discrete Sobolev space defined by W. Hackbusch. The proof of convergence has been provided for both types of boundary conditions when the imaginary part of incident wavenumber is positive.}
\end{abstract}

\setcounter{section}{-1}
\section{Introduction}

Sommerfeld \citep{Sommerfeld1} provided the solution for a two dimensional Helmholtz equation with boundary condition on a half plane of either Dirichlet (diffraction by `soft surface') or Neumann (diffraction by `hard surface'). Many decades later, J. Schwinger and his co-workers \citep{schwinger1968discontinuities,Levine1,Levine2,Levine,MiltonSchwingerbook} formulated the diffraction problems as integral equations of the Wiener--Hopf type \citep{Wiener} and analyzed using the tools discussed by \cite{Paley}. It was found that a formulation of continuous Sommerfeld problems as integral equations of the Wiener--Hopf type has several advantages \citep{Bouwkamp}; for instance, {see the work of}
\cite{Copson} {who originally} applied the `new' method.
{Later} it was found that the integral equation based approach involved various subtle manipulations and {some of} the associated technical details {remained} to be overcome \citep{Heins2}. To tackle some of these issues, the {Sommerfeld} problems have been, thereafter, studied in a well-posed Sobolev space setting \citep{Meister}. The corresponding operator-theoretical approach to the class of diffraction problems, in the presence of a half-plane screen, has been discussed in several distinguished contributions \citep{Speck, Speckbook, Meister}. 
Undoubtedly, these researches have brought mathematical closure on {the continuous} Sommerfeld problems with either a Dirichlet half-plane or a Neumann half-plane.

In recent works of the author, certain {\em discrete analogues} of the continuous Sommerfeld problems have been formulated on a square lattice {where they have been} analyzed using the discrete Wiener--Hopf method \citep{Silbermann}. For example, the results presented by \cite{Bls0, Bls1} are based on Jones' approach \citep{Jones2} (see also \citep{Noble, Jones1}) and lattice model formulation, along with various definitions and notational devices, presented by \cite{Slepyanbook}. On the other hand, \citep{Bls2} and \citep{Bls3} detail a discrete analogue of the integral equation formulation of the continuous Sommerfeld problems employing the square lattice Green's function \citep{Katsura, Eco}. However, in the papers 
\citep{Bls0}--\citep{Bls3}\footnote{with an exception of \citep{Bls2} (\S4.3) where a brief account of the rigorous continuum limit appears for the Neumann (crack) problem, though without a provision of extensive details in the proofs, but with partial announcements of the results stated in this paper}, mostly heuristic asymptotic approximations, {supported by} graphical illustrations, are provided towards the analysis of low frequency approximation of the discrete model, which coincides with, so called, {\em continuum limit}. To supplement this, the present paper provides a rigorous foundation to the asymptotic results for low frequencies stated by \cite{Bls0}--\cite{Bls3}. 

{Naturally, the problem involves two length scales, the wavelength $\lambda=1/k$ of the incident wave and the lattice parameter $b$. However, the wavelength is not necessarily large in the discrete problem, i.e. the incident wave number $k$ (length of the wave vector) is not close to zero. From the traditional continuum point of view the same can be also considered to be a phenomenological effect associated with a possible `resonance' between the incident wave number and the lattice parameter. The ratio $b/\lambda$ (i.e. $bk$) presents itself as a relevant dimensional parameter, a limiting case of which, when it approaches zero, is called `continuum limit' in this paper (borrowing a standard term in the homogenization of discrete media \cite{Blanc, Braides}). In fact, in the context of the assumed discrete structure, either $k\to0$ (long wave wavelength limit\footnote{Due to a simple lattice structure, there is an absence of optical band, hence the zero frequency limit coincides with the long wavelength limit.}) or $b\to0$ (so called continuum limit in homogenization theory \cite{Blanc, Braides}) represent the same limit. \footnote{The additional issue of $kr\to0$, that arises in the continuum model because of the asymptotic nature of analysis in continuous body with sharp edge shaped defect, does not arise in the discrete model. In fact, the near tip field in the discrete model does not possess the same structure as that for the continuum limit (see \cite{Bls2,Bls3}); indeed for fixed $k$, $r\to0$ is not meaningful in the discrete model. However, the issue of $kr\to\infty$ arises in the discrete as well as the continuum model in an analogous manner though the asymptotics are different due to anisotropy of the discrete model \cite{Bls0,Bls1}.}}
In this paper, the continuum limit of the {\em discrete Sommerfeld problems}, posed as discrete Wiener--Hopf problems using the square lattice Green's function \citep{Martin}, is established using certain discrete Sobolev spaces defined by \cite{Hackbusch}. It is established that the exact solution of the discrete Wiener--Hopf equation governing the discrete Sommerfeld diffraction problem converges to the solution of the equivalent continuous Wiener--Hopf equation, governing the corresponding continuous problems in certain discrete Sobolev space of fractional order \citep{Hackbusch, Stevenson}. Indeed, this analysis of {discrete Sommerfeld problems} is also relevant to the scattering of plane polarized electromagnetic waves by a conducting half plane as a result of $5$-point numerical discretization of the two-dimensional Helmholtz equation. The same discretization can also be used in the acoustical counterpart of the problems, where plane waves are supposed to impinge on a soft or hard semi-infinite screen. As the continuous Sommerfeld problems also appear in elastodynamics, in the form of diffraction of elastic shear wave by either a rigid constraint or a crack \citep{Achenbach, Felsen}, the papers \citep{Bls0}--\citep{Bls3}, as well as this paper, are motivated by a discrete analogue of the elastic model, applications of which have a rich history in mechanics of crystals \citep{disloc_maradudin1, Maradudinbook, Slepyan2, Slepyanbook}. 

\subsection{Outline}
A short description of the continuous Wiener--Hopf formulation is provided in the first section for diffraction of a wave incident on a Sommerfeld half plane with Dirichlet boundary condition and Neumann boundary condition. The discrete Wiener--Hopf equation is stated using the square lattice Green's function with discrete Dirichlet and discrete Neumann boundary conditions. A correspondence between the continuous and discrete Wiener--Hopf equations is provided through a suitable choice of notation and scaling. As the main result, it is shown that the continuous and discrete problems yield solutions which can be brought arbitrarily close to each other in a strict mathematical sense using the definition of discrete Sobolev spaces. Concluding remarks close the discussion of this paper, while some additional derivations and expressions appear in two short appendices.

\subsection{Notation}
Let ${\mathbb{R}}$ denote the set of real numbers, ${\mathbb{C}}$ denote the set of complex numbers, and ${\mathbb{Z}}$ denote the set of integers. Let ${{{\mathbb{Z}}^2}}$ denote ${\mathbb{Z}}\times{\mathbb{Z}}$ and ${\mathbb{R}}^2$ denote ${\mathbb{R}}\times{\mathbb{R}}$. The real part, $\Re {z},$ of a complex number ${z}\in{\mathbb{C}}$ is denoted by ${z}_1\in{\mathbb{R}}$ and its imaginary part, $\Im {z}$, is denoted by ${z}_2\in{\mathbb{R}}$ (so that ${z}={z}_1+i{z}_2$). Let $|{z}|$ denote the modulus and $\arg {z}$ denote the argument (with standard branch cut along negative real axis) for ${z}\in{\mathbb{C}}$. Let ${\mathbb{Z}}^+$ denote the set of all non-negative integers and ${\mathbb{Z}}^-$ denote the set of all negative integers. Similarly, ${\mathbb{R}}^+$ denotes non-negative real numbers while ${\mathbb{R}}^-$ denotes negative real numbers. Let $\ell_2$ denote square summable (complex valued) sequences on ${\mathbb{Z}}^-$. Let ${\mathscr{L}}_2(I)$ denote the square summable (complex valued) functions on $I\subset{\mathbb{R}}$ in the sense of Lebesgue. The notation ${\mathit{u}}^{{t}}({\mathit{x}}, +0)$ implies that $\lim_{{{\mathit{y}}}\to+0} {\mathit{u}}^{{t}}({\mathit{x}}, {{\mathit{y}}}) ={\mathit{u}}^{{t}}({\mathit{x}}, +0)$ and $\pd{u}{{{\mathit{y}}}}({\mathit{x}}, +0)$ implies that $\lim_{{{\mathit{y}}}\to+0} \pd{u}{{{\mathit{y}}}}({\mathit{x}}, {{\mathit{y}}}) =\pd{{\mathit{u}}^{{t}}}{{{\mathit{y}}}}({\mathit{x}}, +0).$ Similar interpretations are available for ${\mathit{u}}^{{t}}({\mathit{x}}, -0),$ etc. The discrete (continuous) Fourier transform of a sequence $\{{{\mathtt{u}}}_m\}_{m\in{\mathbb{Z}}}$ (function $u$) is denoted by ${{\mathtt{u}}}^F$ ($u^F$). The continuous and discrete convolutions are denoted by $\ast$, the nature of which is clear from the context. The letter ${{\mathbb{T}}}$ denotes the unit circle (as a counterclockwise contour) in complex plane ${\mathbb{C}}$. The letter ${{z}}$ or ${\upxi}$ is used as a complex variable for the discrete Fourier transform, whereas ${{\xi}}$ is used as a complex variable for the continuous Fourier transform. The letter ${{\mathit{H}}}$ stands for the Heaviside function: ${{\mathit{H}}}(x)=0, x<0$ and ${{\mathit{H}}}(x)=1, x\ge0$. Latin letters ${\mathit{C}}_1, {\mathit{C}}_2,$ etc, denote constants in expressions, inequalities, etc. The square root function, $\sqrt{\cdot}$, has the usual branch cut in the complex plane running from $-\infty$ to $0$. 

Throughout the paper, ``(D)'' denote ``(Dirichlet)'', while ``(N)'' stand for ``(Neumann)''; further, let these two cases have corresponding association with $\pm$, or $\mp$, in context (upper choice is associated with Dirichlet). The notation for other relevant entities is described in the main text.

\section{Wiener--Hopf formulation of continuous Sommerfeld problems}
Following \cite{Noble} (\S2.4), the formulation of the Sommerfeld half-plane diffraction problem in terms of integral equations is considered. It is assumed that the incident wave ${\mathit{u}}^{{i}}({\mathit{x}}, {{\mathit{y}}})$ is a plane wave with {\em wave number} ${\mathit{k}}$ and the {\em angle of incidence}, ${{\mathbb{T}}heta}$ ($0 < {{\mathbb{T}}heta}<\pi$), which the direction of wave propagation normal makes with respect to the positive ${\mathit{x}}$ axis. Following the tradition in scattering theory, the harmonic time dependence of the form $e^{-i{\omega} t}$ has been ignored, and the incident wave is defined by
\beqans
{\mathit{u}}^{{i}}({\mathit{x}}, {{\mathit{y}}})&=&e^{-i({\mathit{x}}{\mathit{k}}_x +{{\mathit{y}}}{\mathit{k}}_y)},\label{cuinc}\\
\text{\rm where }
{\mathit{k}}_x&=&{\mathit{k}}\cos{{\mathbb{T}}heta}\text{\rm and }{\mathit{k}}_y ={\mathit{k}}\sin{{\mathbb{T}}heta}, {\mathit{k}}^2={\mathit{k}}_x^2+{\mathit{k}}_y^2.\label{ckxky}
\eeqans{cuk}
Recall that (also see, equation 5.1.3, 5.2.1 stated by \cite{Harris2}) for a plane wave incident upon a half-plane ${\mathit{x}} < 0, {\mathit{y}} =0$,
\begin{eqn}
{\mathit{u}}^{{t}}({\mathit{x}}, {\mathit{y}})&={\mathit{u}}^{{i}}({\mathit{x}}, {\mathit{y}})+\int_{-\infty}^0 {{\mathcal{G}}}({\mathit{x}}-t, {\mathit{y}})({\mathit{u}}^{{t}}_{{\mathit{y}}}(t, -0)-{\mathit{u}}^{{t}}_{{\mathit{y}}}(t, +0))dt\\
&-\int_{-\infty}^0 \pd{}{{{\mathit{y}}}}{{\mathcal{G}}}({\mathit{x}}-t, {\mathit{y}})({\mathit{u}}^{{t}}(t, +0)-{\mathit{u}}^{{t}}(t, -0))dt,
\label{heins2.1}
\end{eqn}
is a representation of the solution of the traditional Helmholtz equation 
\begin{eqn}
{\Delta} u+{\mathit{k}}^2 u=0,
\label{heins2.2}
\end{eqn}
with a half plane boundary condition at ${\mathit{y}}=+0, {\mathit{y}}=-0, {\mathit{x}}<0$. In \eqref{heins2.2}, the operator ${\Delta}$ denotes the Laplacian, i.e., $\pd{^2}{{\mathit{x}}^2}+\pd{^2}{{\mathit{y}}^2}$, in two spatial dimensions. 
The integral form \eqref{heins2.1} of a solution of \eqref{heins2.2} uses the {\em free space Green's function} \citep{Noble} for the Helmholtz equation \eqref{heins2.2} with fixed (Dirac-delta distributional) source at $(0, {0});$ in particular, ${{\mathcal{G}}}: {\mathbb{R}}^2\to{\mathbb{C}}$ is given by
\begin{eqn}
{{\mathcal{G}}}({\mathit{x}}, {\mathit{y}})=\frac{i}{4}H_0^{(1)}({\mathit{k}} \sqrt{{\mathit{x}}^2+{\mathit{y}}^2}),
\label{cGreen}
\end{eqn}
where $H_0^{(1)}$ is a Hankel function of the first kind \citep{Abram}. The classical boundary conditions of two kinds which Sommerfeld 
\citep{Sommerfeld1, Sommerfeld2} investigated are 
\beqans
{\mathit{u}}^{{t}}({\mathit{x}}, +0) ={\mathit{u}}^{{t}}({\mathit{x}}, -0) =0, {\mathit{x}}\in{\mathbb{R}}^-&&\quad\quad\quad\text{\rm (D)},\label{dirichlet}\\
{\mathit{u}}^{{t}}_{{\mathit{y}}}({\mathit{x}}, +0) ={\mathit{u}}^{{t}}_{{\mathit{y}}}({\mathit{x}}, -0) =0, {\mathit{x}}\in{\mathbb{R}}^-&&\quad\quad\quad\text{\rm (N).}\label{neumann}
\eeqans{cbc}

\begin{remark}
Note that the continuum dispersion relation 
\begin{eqn}
{\omega}^2={\mathit{k}}^2, 
\label{cdisp}
\end{eqn}
for the traditional wave equation is utilized in the Helmholtz equation \eqref{heins2.2} where ${\omega}$ is the {\em frequency} of the incident wave and ${\mathit{k}}$ is its wave number.
In equation \eqref{heins2.1} the incident wave ${\mathit{u}}^{{i}}$ is already incorporated; also included is the Sommerfeld radiation condition when ${\mathit{k}}_2=0$, though this paper assumes ${\mathit{k}}_2=\Im {\mathit{k}}>0$ which obviates the need of it. Since ${\mathit{k}}_2 > 0$, therefore, ${\omega}$ is considered with a positive imaginary part ${\omega}_2$, i.e.,
\begin{eqn}
{\omega}={\omega}_1+i{\omega}_2, {\omega}_2>0.
\label{complexfreq}
\end{eqn}
\end{remark}

Using \eqref{heins2.1}, naturally, each of the two choices of the boundary conditions, \eqref{dirichlet} and \eqref{neumann}, leads to an integral equation of the Wiener--Hopf type for a single unknown for either the discontinuity of the normal derivative of ${\mathit{u}}^{{t}}$ (in the case of Dirichlet boundary condition \eqref{dirichlet}) or the discontinuity of ${\mathit{u}}^{{t}}$ (in the case of Neumann boundary condition \eqref{neumann}). Indeed, when the respective boundary conditions \eqref{dirichlet} and \eqref{neumann} are incorporated, and the even symmetry is invoked for the former scattering problem and odd symmetry for the latter (about ${\mathit{y}}=0$), it follows from \eqref{heins2.1} that
\beqans
0=&{\frac{1}{2}} e^{-i{\mathit{x}} {\mathit{k}}_x}-\int_{-\infty}^0 {{\mathcal{G}}}(|{\mathit{x}}-t|, 0){\mathit{u}}_{{\mathit{y}}}(t, +0)dt, \quad\quad\quad\quad\quad\quad\quad\quad\quad{\mathit{x}}\in{\mathbb{R}}^-&\quad\quad\text{\rm (D)},\label{cWHD}\\
0=&-{\frac{1}{2}} i {\mathit{k}}_ye^{-i{\mathit{x}} {\mathit{k}}_x}+(\pd{^2}{{\mathit{x}}^2}+{\mathit{k}}^2)\int_{-\infty}^0 {{\mathcal{G}}}(|{\mathit{x}}-t|, 0){\mathit{u}}(t, +0)dt, \quad\quad{\mathit{x}}\in{\mathbb{R}}^-&\quad\quad\text{\rm (N),}\label{cWHN}
\eeqans{cWH}
i.e. (2.49) and (2.56) stated by \cite{Noble}, respectively. For the Dirichlet boundary condition \eqref{dirichlet}, ${\mathit{u}}^{{t}}({\mathit{x}}, 0)$ is not known when ${\mathit{x}} > 0,$ while for the Neumann boundary condition \eqref{neumann}, ${\mathit{u}}_{{\mathit{y}}}^{{t}}({\mathit{x}}, 0)$ when ${\mathit{x}} > 0.$ In view of the fact that ${\mathit{k}}$ has a positive imaginary part, the Fourier integral theorem in the complex domain can be applied to the associated integral in both equations \citep{Noble}. 

Each of the equations in \eqref{cWH} is an integral equation of the Wiener--Hopf type as it involves integration on the line ${{\mathit{y}}} =0$, $-\infty < {\mathit{x}} \le0$ and, in particular, the integral kernel is a {\em convolution kernel} \citep{Gohberg}, which is represented by ${\mathpzc{k}}$. In order to allow an application of relevant operator-theoretic results, the equations \eqref{cWHD} and \eqref{cWHN}
can be expressed as
\begin{eqn}
\mathcal{K}{\mathpzc{x}}={\mathpzc{k}}\ast{\mathpzc{x}}&={\mathpzc{f}}, 
\label{cWH2}
\end{eqn}
where, for all ${\mathit{x}}\in{\mathbb{R}}^-$,
\beqans
(\mathcal{K}{\mathpzc{x}})({\mathit{x}})&=&\begin{dcases}
\int_{-\infty}^0 {{\mathcal{G}}}(|{\mathit{x}}-t|, 0){\mathpzc{x}}(t)dt &\text{\rm (D)},\\
(\pd{^2}{{\mathit{x}}^2}+{\mathit{k}}^2)\int_{-\infty}^0 {{\mathcal{G}}}(|{\mathit{x}}-t|, 0){\mathpzc{x}}(t)dt&\text{\rm (N),}
\end{dcases}\label{cKx}\\
{\mathpzc{x}}({\mathit{x}})&=&\begin{cases}
{\mathit{u}}_{{\mathit{y}}}({\mathit{x}})&\quad\quad\quad\quad\quad\quad\quad\quad\quad\quad\quad\quad\quad\text{\rm (D)},\\
{\mathit{u}}({\mathit{x}})&\quad\quad\quad\quad\quad\quad\quad\quad\quad\quad\quad\quad\quad\text{\rm (N)},
\end{cases}\label{cwx}\\
\text{\rm and }{\mathpzc{f}}({\mathit{x}})&=&\begin{cases}
{\frac{1}{2}} e^{-i{\mathit{x}} {\mathit{k}}_x}&\quad\quad\quad\quad\quad\quad\quad\quad\quad\quad\text{\rm (D)},\\
{\frac{1}{2}} i {\mathit{k}}_ye^{-i{\mathit{x}} {\mathit{k}}_x}&\quad\quad\quad\quad\quad\quad\quad\quad\quad\quad\text{\rm (N).}
\end{cases}
\label{cfx}
\eeqans{cconv}
The well known Fourier transform ${f}^F({\xi})=\int_{-\infty}^{+\infty}f({\mathit{x}})e^{-i {\xi} {\mathit{x}}}d{\mathit{x}}, $ of the convolution kernel ${\mathpzc{k}}$, using the Fourier transform of Green's function ${\mathcal{G}}$  in \eqref{cKx},  is given by
\begin{eqn}
{{\mathpzc{k}}}^F({\xi})=\begin{dcases}
\frac{i}{2}\dfrac{1}{\sqrt{{\mathit{k}}^2-{\xi}^2}}&\quad\quad\quad\quad\quad\quad\quad\text{\rm (D)},\\
\frac{i}{2}\sqrt{{\mathit{k}}^2-{\xi}^2}&\quad\quad\quad\quad\quad\quad\quad\text{\rm (N).}
\end{dcases}
\label{cKF}
\end{eqn}
As a standard result, note (for example, see \citep{Stephan, Stephan1, Holm}) that 
\begin{eqn}
&\mathcal{K}: \mathscr{H}^{\mp{\frac{1}{2}}}({\mathbb{R}}^-)\to\mathscr{H}^{\pm{\frac{1}{2}}}({\mathbb{R}}^-),\text{ defined by }\eqref{cWH2}\text{ and }\eqref{cconv}, \\
&\text{ is bijective and continuous for each of the two continuous Sommerfeld problems.}
\label{cKinv}
\end{eqn}
Also see, for example, the elaborate researches contained in \citep{Meister,Speck, Speckbook} in this context, and \citep{Adams, Lions} for the terminology of Sobolev spaces.
\begin{remark}
Apart from the operator-theoretic fact stated above, there exists a technical subtlety in the form of possible exponential increase of ${\mathpzc{f}}({\mathit{x}})$ as ${\mathit{x}}\to-\infty$, but this is ignored in this paper as the focus is not on this issue (see \citep{Bls2} where a subsection is devoted to this technical aspect for the discrete model). 
\label{remexpon}
\end{remark}
\begin{remark}
It follows from the well--known properties of Bessel and Hankel functions (see, for example, \citep{Greenspan}) that ${{\mathcal{G}}}({\mathit{x}}, 0)=\frac{1}{2\pi}\ln\frac{1}{|{\mathit{x}}|}+O(1)\text{ as }{\mathit{x}}\to0,$ so that the kernel ${\mathpzc{k}}$ is either logarithmic (for Dirichlet) or hypersingular (for Neumann). 
\end{remark}

Above facts concerning the integral formulation of the continuous Sommerfeld problems, expanded and carefully accounted in several distinguished works such as \citep{Speck, Speckbook, Meister}, are sufficient for the purpose of this paper. 

\section{Wiener--Hopf formulation of discrete Sommerfeld problems on square lattice}
Consider a $5$-point discretization \citep{Collatz} of the Helmholtz equation \eqref{heins2.2}. Let the resulting two dimensional square grid in ${\mathbb{R}}^2$ be denoted by ${{\mathfrak{S}}}$. Following a treatment of the discretized model as a mechanical model \citep{disloc_maradudin1, Slepyanbook, Bls0, Bls1}, let the displacement of a particle in ${{\mathfrak{S}}}$, indexed by its lattice coordinates $({{\mathtt{x}}}, {{\mathtt{y}}})\in{{\mathfrak{S}}},$ be denoted by ${{\mathtt{u}}}_{{{\mathtt{x}}}, {{\mathtt{y}}}}\in{\mathbb{C}}$. Each `particle' in the lattice ${{\mathfrak{S}}}$ is assumed to interact with its four nearest neighbors in ${{\mathfrak{S}}}$ by linearly elastic identical (massless) bonds with a shear spring constant $1/{{\upepsilon}}^2$. Corresponding to the `discrete' Dirichlet boundary condition, which can be associated with a rigid constraint in a natural manner \citep{Bls1}, the total displacement of each particle located at $({\mathtt{x}}, 0)\in{{{\mathbb{Z}}^2}}$ for all negative integers ${\mathtt{x}}$ is constrained to be zero (see left Fig. \ref{squarelattice_diffraction}). On the other hand, the `discrete' Neumann boundary condition, also naturally associated with the existence of a crack free from external surface forces, is modeled by assuming broken bonds between ${{\mathtt{y}}}=0$ and ${{\mathtt{y}}}=-1$ for all negative integers ${\mathtt{x}}$ in square lattice ${{\mathfrak{S}}}$, as shown in the right part of Fig. \ref{squarelattice_diffraction}. Let 
\begin{eqn}{{\mathit{u}}}^{{\upepsilon}}({\mathit{x}}, {\mathit{y}})={{\mathtt{u}}}_{{\mathtt{x}}, {\mathtt{y}}}\text{\rm with }{\mathit{x}}={\mathtt{x}}{\upepsilon}, {\mathit{y}}=\begin{dcases}
{\mathtt{y}}{\upepsilon}&\quad\quad\quad\text{\rm (D)},\\
({\mathtt{y}}+{\frac{1}{2}}){\upepsilon}&\quad\quad\quad\text{\rm (N),}
\end{dcases} 
\label{mnxy}
\end{eqn}
where $({\mathit{x}}, {\mathit{y}})\in{\mathbb{R}}^2$ are the {\em macroscopic coordinates} corresponding to the lattice coordinates $({\mathtt{x}}, {\mathtt{y}})\in{\mathbb{Z}}^2$ of a particle in ${{\mathfrak{S}}}$ \citep{Bls0, Bls1}. ${{\mathit{u}}}^{{\upepsilon}}$ can be considered as a macroscopic (`continuous') counterpart of the discrete field ${{\mathtt{u}}}^{{\upepsilon}}$ (${\upepsilon}$ denotes the grid spacing, see Fig. \ref{squarelattice_diffraction}).

\begin{figure}[htb!]
\centering
{\includegraphics[width=.4\textwidth]{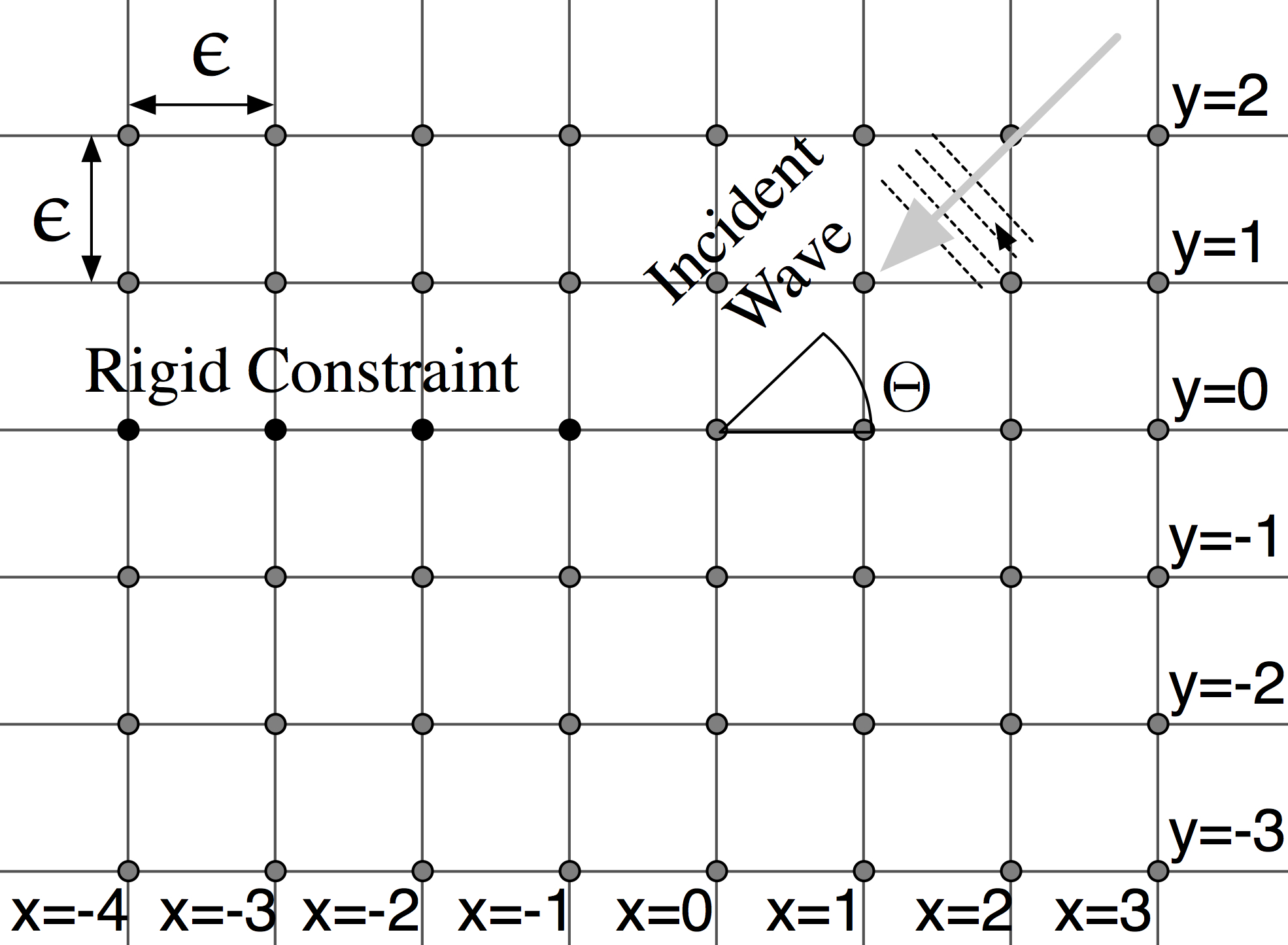}\hspace{.4in}}
{\includegraphics[width=.4\textwidth]{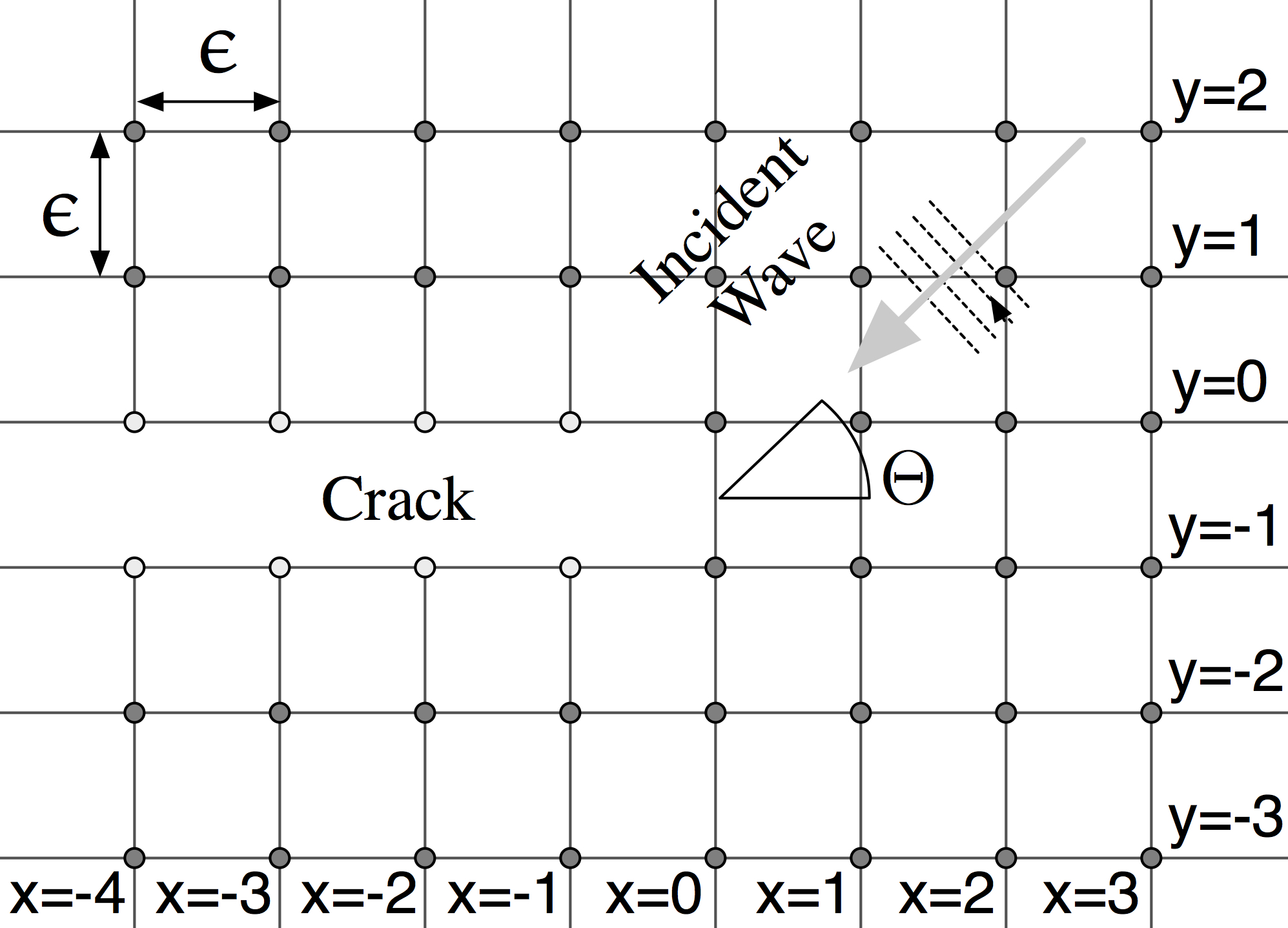}}
\caption{\footnotesize Square lattice ${{\mathfrak{S}}}$ with a semi-infinite rigid constraint (`discrete' Dirichlet boundary condition) in the left figure and with a semi-infinite crack (`discrete' Neumann boundary condition) between ${{\mathtt{y}}}=0$ and ${{\mathtt{y}}}=-1$ in the right figure. An incident lattice wave is also shown, schematically. The intact lattice is shown as solid gray dots. The particles located at the rigid constraint (left) are shown as solid black dots and the particles located at the crack face (right) are shown as solid white dots.}
\label{squarelattice_diffraction}
\end{figure}

On the square lattice model, described thus far, a time harmonic lattice wave is considered incident and the associated diffraction problems, due to the tip of rigid constraint \citep{Bls1, Bls3} and crack \citep{Bls0, Bls2}, are interpreted as {discrete Sommerfeld problems}. The total displacement ${{\mathtt{u}}}^{{t}}={{\mathtt{u}}}^{{i}}+{{\mathtt{u}}}$ of an arbitrary particle in the lattice is a sum of the incident wave displacement ${{\mathtt{u}}}^{{i}}$ and the scattered wave displacement ${{\mathtt{u}}}$ (which includes the reflected wave). On the intact part of the square lattice ${{\mathfrak{S}}}$, ${{\mathtt{u}}}^{{t}}$ satisfies the discrete Helmholtz equation
\beqans
{\triangle}{{\mathtt{u}}}^{{t}}_{{\mathtt{x}}, {\mathtt{y}}}+{\upepsilon}^2{\omega}^2{{\mathtt{u}}}^{{t}}_{{\mathtt{x}}, {\mathtt{y}}}&=&0, \label{dHelmholtz}\\
\text{\rm where }{\triangle}{{\mathtt{u}}}_{{\mathtt{x}}, {\mathtt{y}}}&{:=}&{{\mathtt{u}}}_{{\mathtt{x}}+1, {\mathtt{y}}}+{{\mathtt{u}}}_{{\mathtt{x}}-1, {\mathtt{y}}}+{{\mathtt{u}}}_{{\mathtt{x}}, {\mathtt{y}}+1}+{{\mathtt{u}}}_{{\mathtt{x}}, {\mathtt{y}}-1}-4{{\mathtt{u}}}_{{\mathtt{x}}, {\mathtt{y}}}.\label{dLap}
\label{discretelaplacian}
\eeqans{dHelmholtzfull}
As a discrete counterpart of \eqref{cuinc}, suppose ${{\mathtt{u}}}^{{i}}$ describes an {\em incident lattice wave} with the same frequency ${\omega}$ as in continuous case (also ignoring the factor $e^{-i{\omega} t}$), and a {\em lattice wave vector} $({\upkappa}_x, {\upkappa}_y)\in[-\pi, \pi]^2$, i.e.,
\begin{eqn}
{{\mathtt{u}}}_{{\mathtt{x}}, {\mathtt{y}}}^{{i}}{:=} e^{-i{\upkappa}_x {\mathtt{x}}-i{\upkappa}_y {\mathtt{y}}}, \quad\quad({\mathtt{x}}, {\mathtt{y}})\in{{{\mathbb{Z}}^2}}.
\label{uinc}
\end{eqn}
In terms of the macroscopic coordinates \eqref{mnxy}, the incident lattice wave \eqref{uinc} is expressed as 
\beqans
{{\mathit{u}}}^{{i}}({\mathit{x}}, {\mathit{y}})=\begin{dcases}
e^{-i({\mathit{k}}_x {\mathit{x}}+{\mathit{k}}_y {\mathit{y}})}&\quad\quad\quad\text{\rm (D)},\\
{e^{i{\frac{1}{2}}{\upepsilon}{\mathit{k}}_y}}e^{-i({\mathit{k}}_x {\mathit{x}}+{\mathit{k}}_y {\mathit{y}})}&\quad\quad\quad\text{\rm (N),}
\end{dcases}
\label{dcuinc}\\
\text{\rm with }{\mathit{k}}_x={\upkappa}_x/{\upepsilon}, {\mathit{k}}_y={\upkappa}_y/{\upepsilon}.\label{kxy}
\eeqans{dcuk}
\begin{remark}
It is natural to call $({\mathit{k}}_x, {\mathit{k}}_y)$ as the {\em macroscopic wave vector} of ${{\mathtt{u}}}^{{i}}$ which can be directly identified with \eqref{ckxky}. Thus, in the continuum limit, the incident waves, ${{\mathtt{u}}}^{{i}}$ and ${{\mathit{u}}}^{{i}}$, in discrete and continuous models, respectively, have an immediate correspondence with each other as ${\upepsilon}\to0$. 
\end{remark}

Since the incident lattice wave is a solution of the equation for the intact lattice, the triplet ${\omega}, {\upkappa}_x,$ and ${\upkappa}_y$ satisfy the square lattice dispersion relation, 
\beqans
{\sigma}_{S}({\upkappa}_x, {\upkappa}_y, {\upepsilon}^2{\omega}^2)&=&0,\label{Sdispersion}\\
\text{\rm where }
{\sigma}_{S}({\upkappa}_x, {\upkappa}_y, {\upomega}^2)&{:=}&{\upomega}^2 -4+2\cos{{\upkappa}_x}+2\cos{{\upkappa}_y}, \quad\quad({\upkappa}_x, {\upkappa}_y)\in [-\pi, \pi]^2.
\label{Sdispersionsym}
\eeqans{Sdispersionfull}
Note that, as ${\upepsilon}\to0$, using \eqref{dcuk} the dispersion relation \eqref{Sdispersion} reduces to 
${\omega}^2\simeq {\mathit{k}}_x^2+{\mathit{k}}_y^2=\frac{1}{{\upepsilon}^2}({\upkappa}_x^2+{\upkappa}_y^2),$ which can be interpreted as {macroscopic dispersion relation} and identified with \eqref{cdisp}. 

Recall that the frequency ${\omega}$ of incident lattice wave ${{\mathtt{u}}}^{{i}}$ \eqref{uinc} (same as that of incident continuous wave ${{\mathit{u}}}^{{i}}$ \eqref{cuinc}), is considered as a complex number \eqref{complexfreq}. In the context of this paper, it is assumed that ${\upepsilon}{\omega}_1$ lies in $(0, 2\sqrt{2}),$ and, moreover, since low frequency limit is studied in this paper, it is assumed \cite{Bls0, Bls1} that $0<{\upepsilon}\ll1$ so that it is not close to any non-zero frequency in the exceptional set $\{0, 2, 2\sqrt{2}\}$ \citep{Shaban}. Let, ${\upkappa}$, the {\em lattice wave number} of incident lattice wave ${{\mathtt{u}}}^{{i}}$, be defined by the relation (compare with \eqref{ckxky})
\begin{eqn}
{\upkappa}_x={\upkappa}\cos{{\mathbb{T}}heta}, {\upkappa}_y={\upkappa}\sin{{\mathbb{T}}heta}, \\
{\upkappa}={\upkappa}_1+i{\upkappa}_2, {\upkappa}_1\ge0,
\label{complexk}
\end{eqn}
where ${{\mathbb{T}}heta}\in(-\pi, \pi]$ is the angle of incidence of ${{\mathtt{u}}}^{{i}}$ \eqref{uinc} (same as that of ${\mathit{u}}^{{i}}$). In the same way as ${\upkappa}$ is determined by \eqref{complexk}, its continuous analogue, ${\mathit{k}}$, is defined by ${\mathit{k}}{:=}{\upepsilon}^{-1}{\upkappa},$ so that it  is interpreted as the {\em macroscopic wave number} of incident wave and identified with that defined in \eqref{cuk}. 

In \citep{Bls2, Bls3}, it is shown that the crack and rigid constraint diffraction problems are equivalent to that of inverting a Toeplitz operator using the square lattice Green's function \citep{Katsura, Eco, Martin, Zemla} ${{\mathcal{G}}}^{{\upepsilon}}: {{\mathbb{Z}}^2}\to{\mathbb{C}}$ 
\begin{eqn}
{{\mathcal{G}}}^{{\upepsilon}}_{{{\mathtt{x}}}, {{\mathtt{y}}}}=\frac{1}{4\pi^2} \int_{-\pi}^{\pi} \int_{-\pi}^{\pi} \frac{\cos {{\mathtt{x}}}{{\upxi}}\cos {{\mathtt{y}}}{{\upeta}}}{{\sigma}_{S}({\upxi}, {\upeta}, {\upepsilon}^2{\omega}^2)} d{{\upxi}} d{{\upeta}}, \quad\quad({{\mathtt{x}}}, {{\mathtt{y}}})\in{{{\mathbb{Z}}^2}},
\label{latticegreen3}
\end{eqn}
where the definition of ${\sigma}_{S}$ in \eqref{Sdispersionsym} is used. For the discrete Dirichlet problem, a unique solution can be found \citep{Bls1, Bls3} in terms of the displacement $\{{{\mathtt{u}}}_{l, 1}\}_{l\in{\mathbb{Z}}^-}$ and ${{\mathtt{u}}}_{0, 0}$, while for the discrete Neumann problem, in terms of the displacement $\{{{\mathtt{u}}}_{l, 0}\}_{l\in{\mathbb{Z}}^-}$ \citep{Bls0, Bls2}. For the discrete Dirichlet problem, using the discrete Helmholtz equation for ${{\mathtt{y}}}=1, {\mathtt{x}}\in{\mathbb{Z}}^-$, while for the discrete Neumann problem, using an analogous equation for ${{\mathtt{y}}}=0, {\mathtt{x}}\in{\mathbb{Z}}^-$, a system of equations in the form of a discrete convolution is found \citep{Krein, Gohberg}. Tuning these discrete Wiener--Hopf equations, stated and derived by \cite{Bls2, Bls3} for both discrete Sommerfeld problems, to suit the issue attended in this paper (compare \eqref{cWH2}), they are rewritten as convolution equations of the form
\begin{eqn}
\mathcal{K}^{{\upepsilon}}{\mathpzc{x}}^{{\upepsilon}}={\mathpzc{k}}^{{\upepsilon}}\ast{\mathpzc{x}}^{{\upepsilon}}&={\mathpzc{f}}^{{\upepsilon}}, 
\label{dWH}
\end{eqn}
(with $(\mathcal{K}^{{\upepsilon}}{\mathpzc{x}}^{{\upepsilon}})_{{\mathtt{x}}}
=\sum_{j=-\infty}^{-1}{\mathtt{k}}^{{\upepsilon}}_{{\mathtt{x}}-j}{\mathpzc{x}}^{{\upepsilon}}_j, 
{{\mathtt{x}}}\in{\mathbb{Z}}^-$) where
\beqans
{\mathtt{k}}^{{\upepsilon}}_{j}&=&\begin{dcases}
{\frac{1}{2}}{\upepsilon}{\mathtt{c}}^{{\upepsilon}}_{j} &\quad\quad\quad\quad\quad\quad\quad\quad\text{\rm (D)},\\
-{\upepsilon}^{-1}{\mathtt{c}}^{{\upepsilon}}_{j} &\quad\quad\quad\quad\quad\quad\quad\quad\text{\rm (N),}
\end{dcases} j\in{\mathbb{Z}},\label{dWHk}\\
{\mathtt{c}}^{{\upepsilon}}_j&=&\begin{dcases}
\delta_{j, 0}-2{{\mathcal{G}}}^{{\upepsilon}}_{j, 1} &\quad\quad\quad\text{\rm (D)},\\
\delta_{j, 0}-2({{\mathcal{G}}}^{{\upepsilon}}_{j, 1}-{{\mathcal{G}}}^{{\upepsilon}}_{j, 0})&\quad\quad\quad\text{\rm (N),}
\end{dcases}
\text{, moreover, }{\mathtt{c}}^{{\upepsilon}}_j={\mathtt{c}}^{{\upepsilon}}_{-j}, j\in{\mathbb{Z}}^-{\mathit{u}}p\{0\},\label{dWHc}\\
{\mathpzc{x}}^{{\upepsilon}}_j&=&\begin{dcases}
{\upepsilon}^{-1}{{\mathtt{v}}}_{j, 1}, {{\mathtt{v}}}_{j, 1}{:=}{{\mathtt{u}}}_{j, 1}-{{\mathtt{u}}}_{j, 0}&\quad\quad\text{\rm (D)},\\
{{\mathtt{u}}}_{j, 0}&\quad\quad\text{\rm (N),}
\end{dcases}j\in{\mathbb{Z}}^-,\label{dWHw}\\
{\mathpzc{f}}^{{\upepsilon}}_{{\mathtt{x}}}&=&\begin{dcases}
\begin{rcases}
{\frac{1}{2}} {{\mathcal{G}}}^{{\upepsilon}}_{{\mathtt{x}}+1, 1}{{\mathtt{u}}}^{{t}}_{0, 0}+{\frac{1}{2}}{{\mathtt{u}}}^{{i}}_{{\mathtt{x}}, 0}&\\
+{\frac{1}{2}}\sum_{j=-\infty}^{-1}{\frac{1}{2}} (\delta_{{\mathtt{x}}, j}-2{\upepsilon}^{-1}{\mathtt{k}}^{{\upepsilon}}_{{{\mathtt{x}}}-j})({{\mathtt{u}}}^{{i}}_{j, 1}+{{\mathtt{u}}}^{{i}}_{j, -1}-2{{\mathtt{u}}}^{{i}}_{j, 0})&
\end{rcases}
&\text{\rm (D)},\\
-{\upepsilon}^{-1}\sum_{j=-\infty}^{-1}{\frac{1}{2}} (\delta_{{\mathtt{x}}, j}-(-{\upepsilon}){\mathtt{k}}^{{\upepsilon}}_{{{\mathtt{x}}}-j})({{\mathtt{u}}}^{{i}}_{j, 0}-{{\mathtt{u}}}^{{i}}_{j, -1})&\text{\rm (N),}
\end{dcases} {{\mathtt{x}}}\in{\mathbb{Z}}^-\label{dWHf}
\eeqans{dWHfull}
The decoration ${\upepsilon}$ connotes the length scale dependence of the corresponding entity; in order to avoid notational difficulties, the dependence has been suppressed for a few symbols, whose nature is clear from the context.

The existence and uniqueness of the exact solution, when ${\omega}_2>0$, for the discrete Sommerfeld problems \eqref{dWH} is established by \cite{Bls2, Bls3}, using the fundamental results of \cite{Krein, Widom1}, 
where it is shown that the 
relevant Toeplitz operator is invertible on $\ell_2$. The latter holds on a possibly weighted space, depending on the incidence angle, but that aspect is not touched 
in this paper as stated in {the Remark \ref{remexpon}.}
\begin{remark}
The continuum limit of the discrete Helmholtz equation can be regarded as the classical two dimensional Helmholtz equation \citep{Hackbusch}, and, not surprisingly, it is found by \cite{Bls0, Bls1} that the integral form of the discrete solution (in both discrete Sommerfeld problems) asymptotically approaches the continuum solution as ${\upepsilon}\to0$. Partial announcement of the same for the Neumann (crack) problem has appeared in \citep{Bls2}. 
In the remaining paper, 
these asymptotic results are provided a detailed rigorous foundation based on the discrete and continuous Sobolev spaces of fractional order. 
The particular choice of placement of ${\upepsilon}$ and its exponents in \eqref{dWHfull} shall become justified in the next section. 
\end{remark}

For ${\omega}_2>0$ and for all ${z}\in{\mathbb{T}}$, by an application of \eqref{latticegreen3} and simplification of the relevant contour integral on ${{\mathbb{T}}}$, it has been found by \cite{Bls2, Bls3} that
\begin{eqn}
\sum_{-\infty}^{+\infty}{\mathtt{c}}^{{\upepsilon}}_{{\mathtt{x}}}{z}^{-{{\mathtt{x}}}}=\begin{dcases}\frac{1}{{{\mathpzc{L}}}_{{c}}({z})}&\text{\rm (D)},\\{{\mathpzc{L}}}_{{k}}({z})&\text{\rm \rm (N),}\end{dcases}
\label{FCL}
\end{eqn}
where
\beqan
{{\mathpzc{L}}}_{{k}}({{z}})&{:=}&\frac{{{\mathpzc{h}}}({{z}})}{{{\mathpzc{r}}}({{z}})}, {{\mathpzc{h}}}({{z}}){:=}\sqrt{{{\mathpzc{H}}}({{z}})}, {{\mathpzc{r}}}({{z}}){:=}\sqrt{{{\mathpzc{R}}}({{z}})}, \quad\quad{{z}}\in{\mathbb{C}}\setminus{\mathscr{B}},\label{Lk}\\
{{\mathpzc{L}}_{{c}}}&{:=}&\frac{{{\mathpzc{r}}}{{\mathpzc{h}}}}{{{\mathpzc{Q}}}}\text{\rm on }{\mathbb{C}}\setminus{\mathscr{B}},\label{Lc}\\
{{\mathpzc{H}}}({{z}})&{:=}&{{\mathpzc{Q}}}({{z}})-2, {{\mathpzc{R}}}({{z}}){:=}{{\mathpzc{Q}}}({{z}})+2, {{\mathpzc{Q}}}({{z}}){:=}4-{{z}}-{{z}}^{-1}-{\upepsilon}^2{\omega}^2, {{z}}\in{\mathbb{C}},
\label{HR}
\eeqan
and ${\mathscr{B}}$ denotes the union of branch cuts for ${\mathpzc{L}}_{{k}}$, borne out of the chosen branch \eqref{branch} for ${{\mathpzc{h}}}$ and ${{\mathpzc{r}}}$ such that for all ${{z}}\in{\mathbb{C}}\setminus{\mathscr{B}}$, 
\begin{eqn}
-\pi<\arg {{\mathpzc{H}}}({{z}}) <\pi, \Re {{\mathpzc{h}}}({{z}})>0, \Re {{\mathpzc{r}}}({{z}})>0, {\text{\rm sgn}} \Im {{\mathpzc{h}}}({{z}})={\text{\rm sgn}} \Im {{\mathpzc{r}}}({{z}}).
\label{branch}
\end{eqn}
Some related details are also briefly recapitulated in Appendix \ref{appL} for reader's convenience.

\begin{figure}[t]\centering
{\includegraphics[width=.9\textwidth]{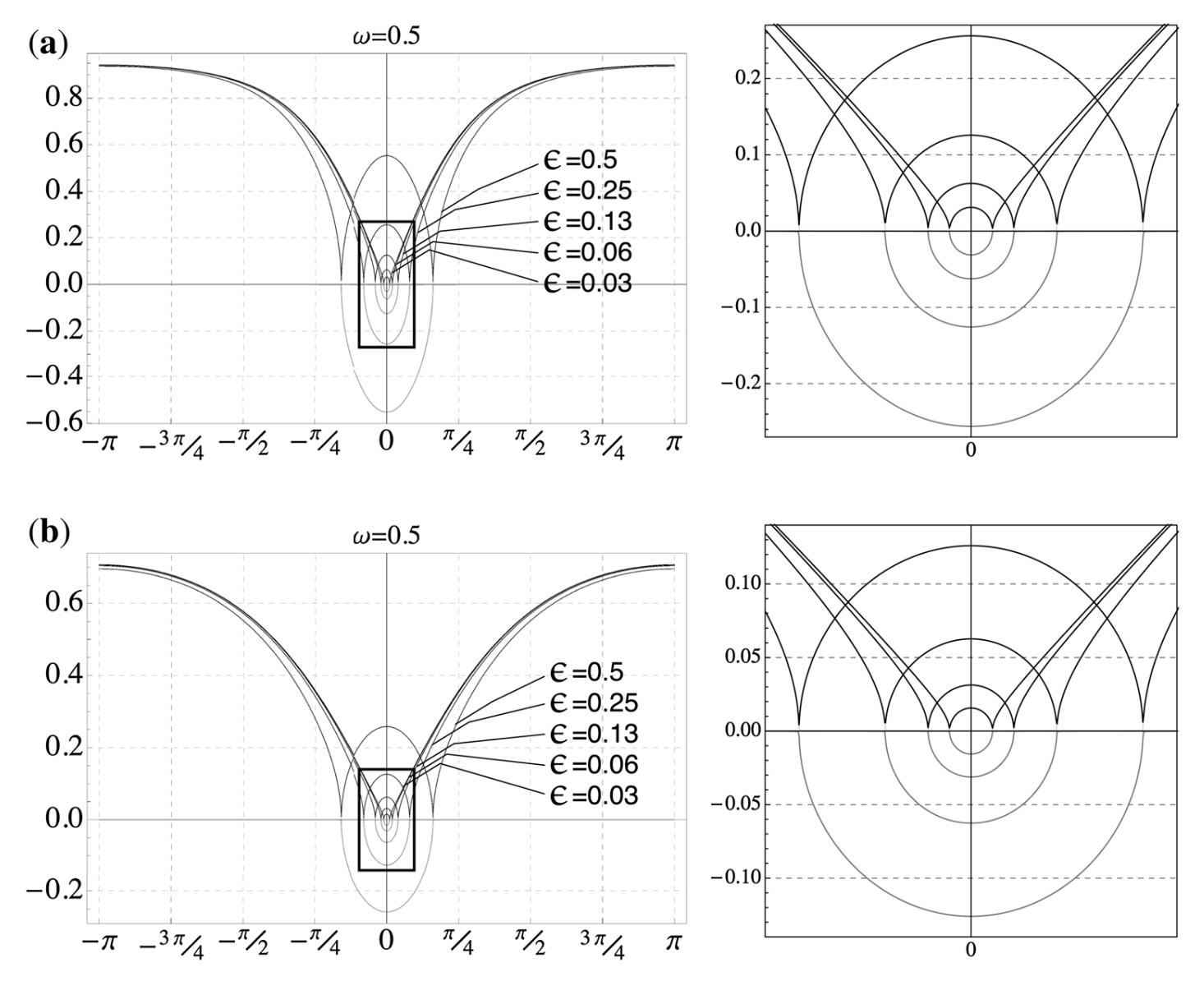}}
\caption{\footnotesize {The kernel ${{\mathpzc{L}}}_{{c}}({z})$ and ${{\mathpzc{L}}}_{{k}}({z})$, respectively, for the (a) rigid constraint and (b) crack in square lattice. Note that ${\omega}=0.5$ (constant for all plots). Light gray denotes real part, gray denotes imaginary part, and black denotes the modulus (on the vertical axis). The horizontal axis corresponds to ${\upxi}$ (with ${z}=e^{-i{\upxi}}$). The right plots present a zoomed-out part of the plot on the left. }}
\label{squareslitcrack_kernelcontlim}\end{figure}

\begin{remark}
By the Fourier series expansion \eqref{FCL} of $1/{{\mathpzc{L}}}_{{c}}$ in the discrete Dirichlet boundary condition \eqref{dWHk}, and the case of Dirichlet condition $(D)$ in \eqref{cKF}, it follows that
\begin{eqn}
\sum_{-\infty}^{+\infty}{\mathtt{k}}^{{\upepsilon}}_{{\mathtt{x}}}{z}^{-{{\mathtt{x}}}}={\frac{1}{2}}{\upepsilon}\frac{1}{{{\mathpzc{L}}}_{{c}}({z})}&={\frac{1}{2}}{\upepsilon}\frac{{\mathpzc{Q}}({z})}{{\mathpzc{h}}({z}){\mathpzc{r}}({z})}\sim{\frac{1}{2}}{\upepsilon}\frac{1}{\sqrt{{\upxi}^2-{\upepsilon}^2{\omega}^2}}={{\mathpzc{k}}}^F({\xi}).
\label{kernellimslit}
\end{eqn}
Correspondingly, for the Neumann boundary condition \eqref{dWHk}, using ${{\mathpzc{L}}}_{{k}}$ in \eqref{FCL} and the corresponding case of Neumann $(N)$ in \eqref{cKF},
\begin{eqn}
\sum_{-\infty}^{+\infty}{\mathtt{k}}^{{\upepsilon}}_{{\mathtt{x}}}{z}^{-{{\mathtt{x}}}}=-{\upepsilon}^{-1}{{\mathpzc{L}}_{{k}}}({z})=-{\upepsilon}^{-1}\frac{{\mathpzc{h}}({z})}{{\mathpzc{r}}({z})}\sim-{\upepsilon}^{-1}{\frac{1}{2}}\sqrt{{\upxi}^2-{\upepsilon}^2{\omega}^2}={{\mathpzc{k}}}^F({\xi}),
\label{kernellimcrack}
\end{eqn}
as ${\upepsilon}\to0,$ with ${z}=e^{-i{\upxi}}, {\upxi}={\upepsilon}{\xi}$. Note that, in the last term for each case of boundary condition in \eqref{dWHf}, ${{\mathtt{u}}}^{{i}}_{j, 1}+{{\mathtt{u}}}^{{i}}_{j, -1}-2{{\mathtt{u}}}^{{i}}_{j, 0}\sim-{\upepsilon}^2{\mathit{k}}_y^2{{\mathtt{u}}}^{{i}}_{j, 0}\to0$ and $-{\frac{1}{2}}{\upepsilon}^{-1}({{\mathtt{u}}}^{{i}}_{j, 0}-{{\mathtt{u}}}^{{i}}_{j, -1})\sim 
i{\frac{1}{2}}{\mathit{k}}_y{{\mathtt{u}}}^{{i}}_{j, 0}$ as ${\upepsilon}\to0$ (corresponding rigorous statements also appear below). {Fig. \ref{squareslitcrack_kernelcontlim} demonstrates an instance of the results stated by \eqref{kernellimslit} and \eqref{kernellimcrack}.}
\end{remark}

\begin{figure}[t]\centering
{\includegraphics[width=\textwidth]{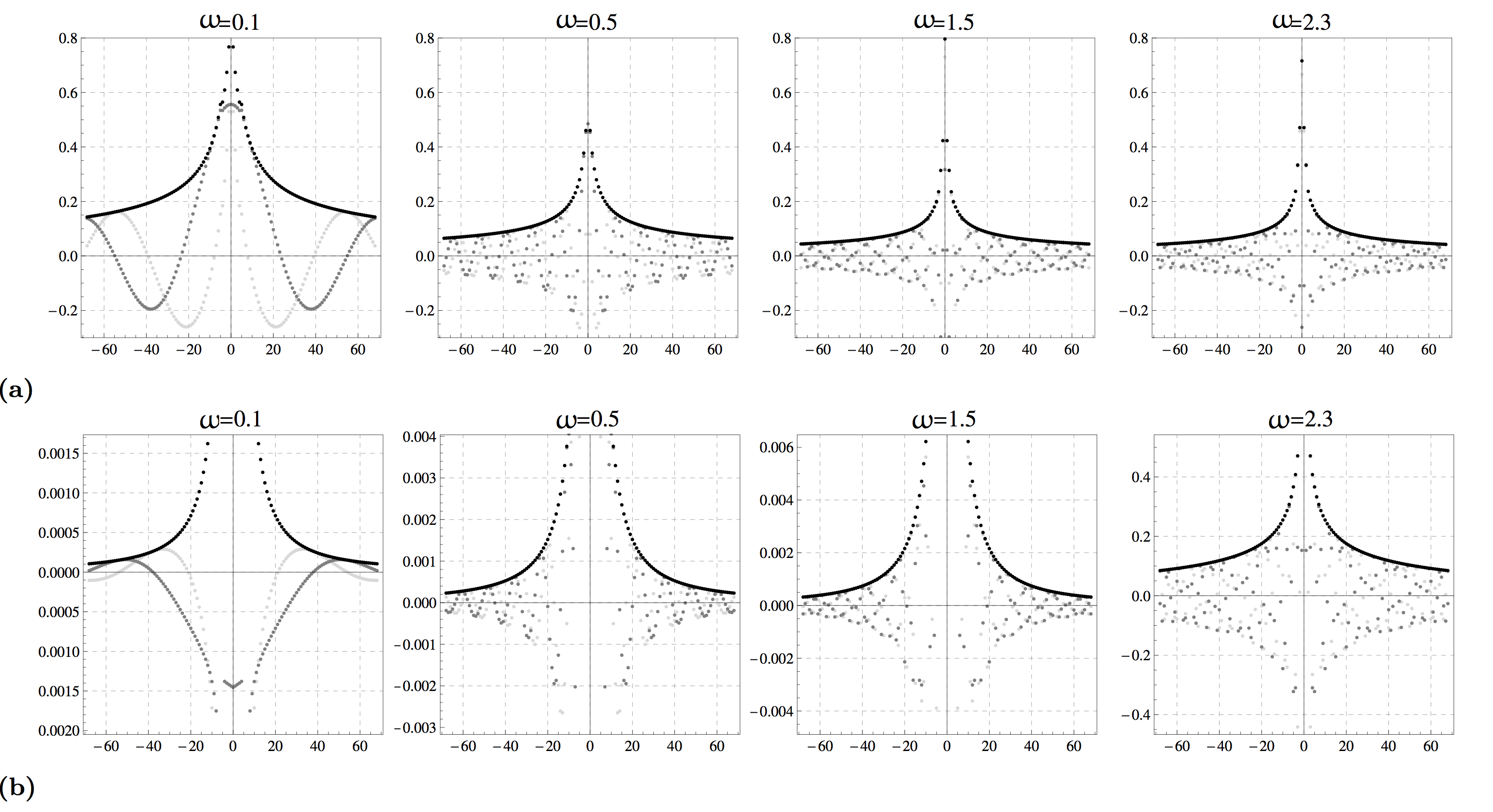}}
\caption{\footnotesize {Fourier coefficients $\{{\mathtt{c}}_j\}_{j\in{\mathbb{Z}}}$ of the kernel ${{\mathpzc{L}}}^{-1}_{{c}}({z})$ and ${{\mathpzc{L}}}_{{k}}({z})$, respectively, as stated by \eqref{FCL} for the (a) Dirichlet condition (rigid constraint) and (b) Neumann condition (crack) in square lattice. Note that ${\upomega}={\upepsilon}{\omega}$ so that ${\upepsilon}$ can be calculated for a given ${\omega}$ (by treating it as a constant for all plots). Black dots represent the modulus and gray dots represent the imaginary part while dotted black curves represent the real part (on the vertical axis). The horizontal axis corresponds to $j$.}}
\label{squareslitcrack_kernelFouriercoefficients}\end{figure}

\begin{figure}[t]\centering
{\includegraphics[width=\textwidth]{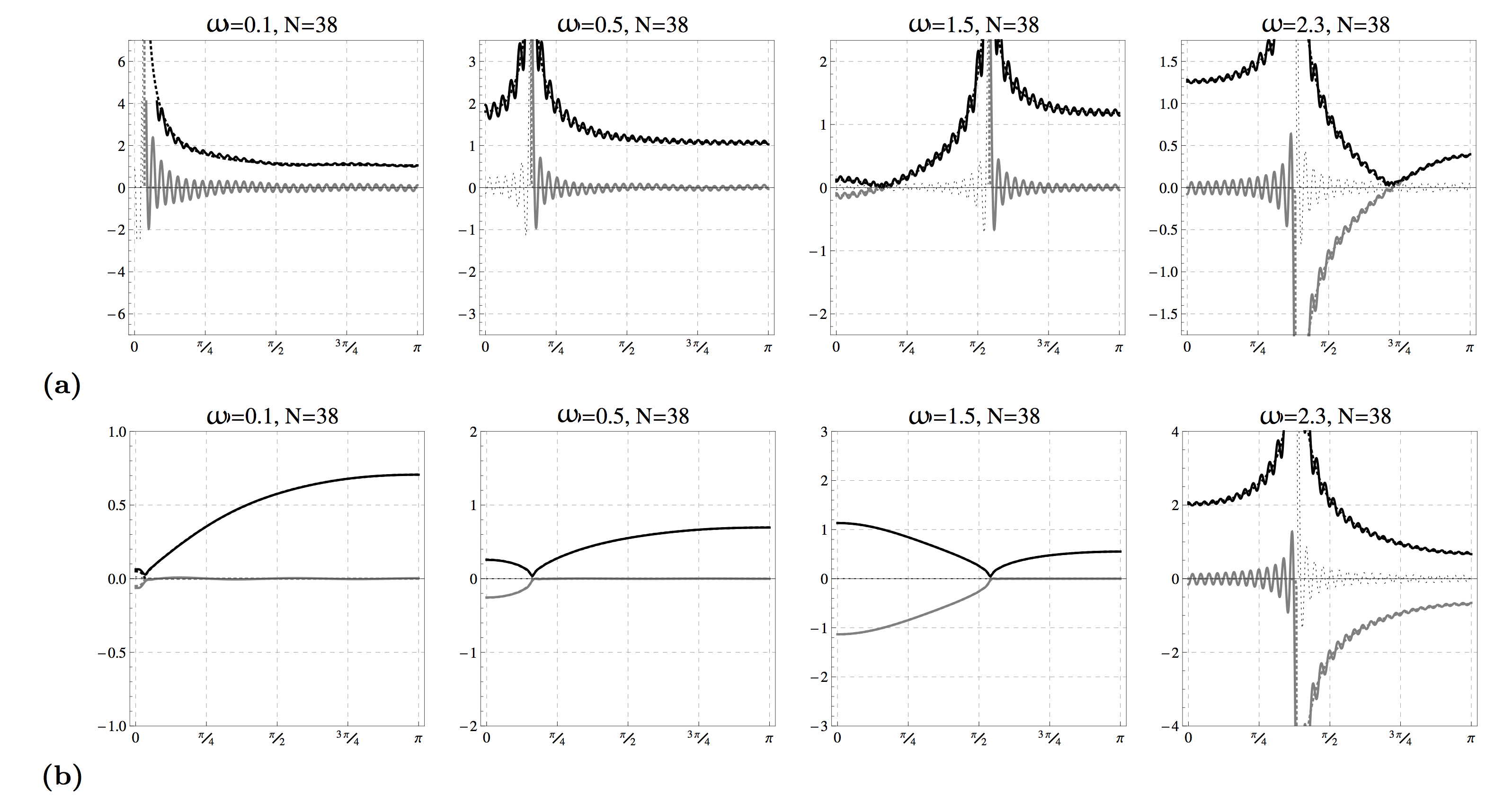}}
\caption{\footnotesize {Fourier partial sum $\sum\nolimits_{-N}^{+N}{\mathtt{c}}_j{z}^{-j}, {z}\in{{\mathbb{T}}}$ and the kernel ${{\mathpzc{L}}}^{-1}_{{c}}({z})$ and ${{\mathpzc{L}}}_{{k}}({z})$, respectively, as stated by \eqref{FCL} for the (a) Dirichlet condition (rigid constraint) and (b) Neumann condition (crack) in square lattice corresponding to Fig. \ref{squareslitcrack_kernelFouriercoefficients}. Black curves represent the modulus and gray curves represent the imaginary part while dotted black curves represent the real part (on the vertical axis). The horizontal axis corresponds to ${\upxi}$ (with ${z}=e^{-i{\upxi}}$).}}
\label{squareslitcrack_kernelFourier}\end{figure}

In the foregoing context, the definitions stated by \cite{Hackbusch, Stevenson} are used to provide an operator theoretic formulation of \eqref{dWH} in the spirit of \eqref{cKinv}. Throughout this paper, the discrete spaces $\mathscr{H}^s= \mathscr{H}^s({\upepsilon}{\mathbb{Z}})$ for $s=+{\frac{1}{2}}$ and $-{\frac{1}{2}}$ are considered. See Appendix \ref{appS} for some more details on discrete {Sobolev} spaces \citep{Hackbusch}. The continuous counterparts are explicitly indicated such as $\mathscr{H}^{-{\frac{1}{2}}}({\mathbb{R}}^-)$. A straightforward results is the following
\begin{lemma}
Let 
$\mathcal{K}_{{\upepsilon}}:\mathscr{H}^{\mp{\frac{1}{2}}} \to \mathscr{H}^{\pm{\frac{1}{2}}}$
be a convolution operator with kernel ${\mathtt{k}}_{{\upepsilon}}$. If ${\mathtt{k}}_{{\upepsilon}}$ is such that ${\varpi^{\pm1}_{{\upepsilon}}}{\mathtt{k}}_{{\upepsilon}}^F\in\mathscr{C}^0(I)$, then $\mathcal{K}_{{\upepsilon}}$ is a bounded operator.
\label{lemma1} 
\end{lemma} 
\begin{proof}
Suppose ${f}_{{\upepsilon}}\in\mathscr{H}^{\mp{\frac{1}{2}}}, {g}_{{\upepsilon}}\in\mathscr{H}^{\pm{\frac{1}{2}}}$ such that 
$$
({\mathtt{k}}_{{\upepsilon}}\ast {f}_{{\upepsilon}})({{\mathtt{x}}}{\upepsilon})=\sum_{j\in{\mathbb{Z}}}{\mathtt{k}}_{{\upepsilon}}(({\mathtt{x}}-j){\upepsilon}){f}_{{\upepsilon}}(j{\upepsilon})={g}_{{\upepsilon}}({{\mathtt{x}}}{\upepsilon}), {{\mathtt{x}}}\in{\mathbb{Z}}.
$$
Let 
\begin{eqn}
{\mathit{C}}_{{\upepsilon}}=\max_{{\upxi}\in I}|{\varpi^{\pm1}_{{\upepsilon}}}({\upxi}){\mathtt{k}}_{{\upepsilon}}^F({\upxi})|.
\label{Ceps}
\end{eqn}
Then,
\begin{eqn}
\norm{{g}_{{\upepsilon}}}_{\pm{\frac{1}{2}}}&=(\frac{{\upepsilon}}{2\pi})^{{\frac{1}{2}}}\norm{\varpi_{{\upepsilon}}^{\pm{\frac{1}{2}}}{g}_{{\upepsilon}}^F}_{{\mathscr{L}}_2(I)}=(\frac{{\upepsilon}}{2\pi})^{{\frac{1}{2}}}\norm{(\varpi_{{\upepsilon}}^{\pm1}{\mathtt{k}}_{{\upepsilon}}^F)(\varpi_{{\upepsilon}}^{\mp{\frac{1}{2}}}{f}_{{\upepsilon}}^F)}_{{\mathscr{L}}_2(I)}\\
&\le (\frac{{\upepsilon}}{2\pi})^{{\frac{1}{2}}}{\mathit{C}}_{{\upepsilon}}\norm{\varpi_{{\upepsilon}}^{\mp{\frac{1}{2}}}{f}_{{\upepsilon}}^F}_{{\mathscr{L}}_2(I)}={\mathit{C}}_{{\upepsilon}}\norm{{f}_{{\upepsilon}}}_{\mp{\frac{1}{2}}}.
\label{DS1}
\end{eqn}
\qed
\end{proof}
{
\begin{remark}
An example for the hypothesis on ${\mathtt{k}}_{{\upepsilon}}$ such that ${\varpi^{\pm1}_{{\upepsilon}}}{\mathtt{k}}_{{\upepsilon}}^F\in\mathscr{C}^0(I)$, corresponding to Lemma \ref{lemma1}, is ${\mathtt{k}}^{{\upepsilon}}_{{\mathtt{x}}}={\upepsilon}^{\pm1}2^{-|{\mathtt{x}}|}$. See Fig. \ref{squareslitcrack_kernelFouriercoefficients} for a graphical illustration of the coefficients in the context of this paper; Fig. \ref{squareslitcrack_kernelFourier} presents the corresponding Fourier partial sums for comparison with each kernel.
\end{remark}
}

In continuation to above statement, as a stronger statement, the following also holds.
\begin{lemma}
The operator $\mathcal{K}_{{\upepsilon}}$ is uniformly bounded as ${\upepsilon}\to0$ for both discrete Sommerfeld problems.
\label{lemma2} 
\end{lemma} 
\begin{proof}
As ${\upepsilon}\to0$ ($\ne0$), it is easy to see that ${\mathit{C}}_{{\upepsilon}}$ in the proof of Lemma \ref{lemma1} becomes independent of ${\upepsilon}$ for the choice of discrete convolution kernel ${\mathtt{k}}={\mathtt{k}}^{{\upepsilon}}$ for both boundary conditions \eqref{dWHk}. Indeed, for the discrete Dirichlet boundary condition $(D)$ in \eqref{dWHk}, and also using its counterpart in \eqref{FCL},
\begin{eqn}
{\mathit{C}}_{{\upepsilon}}&=\max_{{\upxi}\in I}|\varpi_{{\upepsilon}}({\upxi}){\frac{1}{2}}{\upepsilon}\frac{1}{{{\mathpzc{L}}}_{{c}}({z})}|\\
&=\max\bigg\{\max_{{\upxi}\in I}|\varpi_{{\upepsilon}}({\upxi})({\frac{1}{2}}{\upepsilon}\frac{1}{\sqrt{{\upxi}^2-{\upepsilon}^2{\omega}^2}}+o({\upepsilon}))|, |\frac{2}{{\upepsilon}}{\frac{1}{2}}{\upepsilon}\frac{1}{{{\mathpzc{L}}}_{{c}}(-1)}|\bigg\}
\le {\mathit{C}}\frac{1}{{\omega}_2}.
\label{DS12}
\end{eqn}
Similarly, for the discrete Neumann boundary condition $(N)$ in \eqref{dWHk}, using \eqref{FCL},
\begin{eqn}
{\mathit{C}}_{{\upepsilon}}&=\max_{{\upxi}\in I}|\varpi_{{\upepsilon}}^{-1}({\upxi})(-{\upepsilon}^{-1}{{\mathpzc{L}}_{{k}}}({z}))|\\
&=\max\bigg\{\max_{{\upxi}\in I}|\varpi_{{\upepsilon}}^{-1}({\upxi})(-{\upepsilon}^{-1}{\frac{1}{2}}\sqrt{{\upxi}^2-{\upepsilon}^2{\omega}^2}+O(1))|, |{\frac{1}{2}}{{\mathpzc{L}}}_{{k}}(-1)|\bigg\}\le {\mathit{C}}.
\label{DS2}
\end{eqn}
\qed
\end{proof}
For the discrete Sobolev space formulation of \eqref{dWH}, it is also required that ${\mathpzc{f}}^{{\upepsilon}}$ in \eqref{dWHf} lies in the appropriate discrete Sobolev space. Using ${z}=e^{-i{\upxi}}, {z}_{{P}}=e^{-i{\upkappa}_x},$ and adding zeros corresponding to the index ${\mathtt{x}}\in{\mathbb{Z}}^+$, for the discrete Dirichlet boundary condition in \eqref{dWHf},
\begin{eqn}
\varpi_{{\upepsilon}}^{+{\frac{1}{2}}}({\upxi})({\mathpzc{f}}^{{\upepsilon}})^F&=\varpi_{{\upepsilon}}^{+{\frac{1}{2}}}({\upxi})(\frac{1}{4}({\mathtt{c}}_0-(\frac{1}{{{\mathpzc{L}}}_{{c}}})_-{z}){{\mathtt{u}}}^{{t}}_{0, 0}+{\frac{1}{2}}\frac{1}{{z}_{{P}}{z}^{-1}-1}(1-2\sin^2{\frac{1}{2}}{\upkappa}_y(1-\frac{1}{{{\mathpzc{L}}}_{{c}}({z})}))),
\label{DSf1}
\end{eqn}
where $(\frac{1}{{{\mathpzc{L}}}_{{c}}})_-=\sum_{{\mathtt{x}}\in{\mathbb{Z}}^-}{\mathtt{c}}_{{\mathtt{x}}}{z}^{-{\mathtt{x}}}$ is analytic inside a disk of radius $e^{{\upkappa}_2}>1$ due to far field expansion of the square lattice Green's function \citep{Katsura, Martin}. In particular, $(\frac{1}{{{\mathpzc{L}}}_{{c}}})_-$ is bounded (since it is continuous) on ${\mathbb{T}}.$ Note that $\frac{1}{{{\mathpzc{L}}}_{{c}}}=(\frac{1}{{{\mathpzc{L}}}_{{c}}})_++(\frac{1}{{{\mathpzc{L}}}_{{c}}})_-$ is an additive Wiener--Hopf factorization of $\frac{1}{{{\mathpzc{L}}}_{{c}}}.$ Similarly, for the discrete Neumann boundary condition in \eqref{dWHf},
\begin{eqn}
\varpi_{{\upepsilon}}^{-{\frac{1}{2}}}({\upxi})({\mathpzc{f}}^{{\upepsilon}})^F&=-\varpi_{{\upepsilon}}^{-{\frac{1}{2}}}({\upxi}){\upepsilon}^{-1}(1-e^{i{\upkappa}_y})(1-{{\mathpzc{L}}}_{{k}}({z})){\frac{1}{2}}\frac{1}{{z}_{{P}}{z}^{-1}-1}.
\label{DSf2}
\end{eqn}
Since ${\mathpzc{h}}$ has a zero close to ${z}=1$ for small ${\upepsilon}$, the maximum value of $|(\frac{1}{{{\mathpzc{L}}}_{{c}}})_-|$ occurs near $0$. In fact, as ${\upepsilon}\to0$, using ${\upxi}={\upepsilon}{\xi}$, and the far field expansion of Green's function, it is easy to see that $(\frac{1}{{{\mathpzc{L}}}_{{c}}})_-({\upxi})=O(\frac{1}{{\upepsilon}})\sim\frac{{\mathit{C}}}{{\upxi}+{\upkappa}}+o(\frac{1}{{\upepsilon}})$. 

Around ${\upxi}=0$, in \eqref{DSf1}, it is found that 
\begin{eqn}
\varpi_{{\upepsilon}}^{+{\frac{1}{2}}}({\upxi})({\mathpzc{f}}^{{\upepsilon}})^F\sim(1+{\xi}^2)^{1/4}(O(\frac{1}{{\upepsilon}})\frac{1}{{\xi}+{\mathit{k}}}+{\frac{1}{2}}\frac{1}{i({\xi}-{\mathit{k}}_x){\upepsilon}}(1-{\frac{1}{2}}{\upepsilon}^2{\mathit{k}}_y^2(1-\frac{1}{{\upepsilon}\sqrt{{\xi}^2-{\omega}^2}})))\sim O(\frac{1}{{\upepsilon}}),
\label{DSf01}
\end{eqn}
and in \eqref{DSf2}
\begin{eqn}
\varpi_{{\upepsilon}}^{-{\frac{1}{2}}}({\upxi})({\mathpzc{f}}^{{\upepsilon}})^F\sim-(1+{\xi}^2)^{-1/4}({\upepsilon}^{-1}(-1)i{\upepsilon}{\mathit{k}}_y)(1-{\frac{1}{2}}{\upepsilon}\sqrt{{\xi}^2-{\omega}^2}){\frac{1}{2}}\frac{1}{i({\xi}-{\mathit{k}}_x){\upepsilon}}\sim O(\frac{1}{{\upepsilon}}).
\label{DSf02}
\end{eqn}
\begin{lemma}
${\mathpzc{f}}^{{\upepsilon}}\in{\mathscr{H}^{\pm{\frac{1}{2}}}}$. 
\label{lemma3} 
\end{lemma} 
\begin{remark}
Recall the last sentence of the notation, so that, in above lemma, $+$ refers to discrete Dirichlet (D) and $-$ refers to discrete Neumann (N), in \eqref{dWHf}.
\end{remark}
\begin{proof}
Using \eqref{DSf01},
\begin{eqn}
\norm{{\mathpzc{f}}^{{\upepsilon}}}_{+{\frac{1}{2}}}^2&=(\frac{{\upepsilon}}{2\pi})\norm{\varpi_{{\upepsilon}}^{+{\frac{1}{2}}}{\mathpzc{f}}_{{\upepsilon}}^F}^2_{{\mathscr{L}}_2(I)}=(\frac{{\upepsilon}}{2\pi})\int_{-\pi}^{+\pi}|\varpi_{{\upepsilon}}^{+{\frac{1}{2}}}{\mathpzc{f}}_{{\upepsilon}}^F|^2({\upxi})d{\upxi}\\
&=(\frac{{\upepsilon}}{2\pi})(\int_{|{\upxi}|>\sqrt{{\upepsilon}}}|\varpi_{{\upepsilon}}^{+{\frac{1}{2}}}{\mathpzc{f}}_{{\upepsilon}}^F|^2({\upxi})d{\upxi}+\int_{-\sqrt{{\upepsilon}}}^{+\sqrt{{\upepsilon}}}|\varpi_{{\upepsilon}}^{+{\frac{1}{2}}}{\mathpzc{f}}_{{\upepsilon}}^F|^2({\upxi})d{\upxi})\\
&\le {\upepsilon}({\mathit{C}}_1\frac{1}{{{\upepsilon}}}+{\upepsilon}\int_{-1/\sqrt{{\upepsilon}}}^{+1/\sqrt{{\upepsilon}}}|\varpi_{{\upepsilon}}^{+{\frac{1}{2}}}{\mathpzc{f}}_{{\upepsilon}}^F|^2({\upepsilon}{\xi})d{\xi})\\
&\le {\upepsilon}({\mathit{C}}_1\frac{1}{{{\upepsilon}}}+{\mathit{C}}_2\int_{-1/\sqrt{{\upepsilon}}}^{+1/\sqrt{{\upepsilon}}}\frac{1}{\sqrt{1+{\xi}^2}}d{\xi})\\&\le{\mathit{C}}(1+{\upepsilon}\sinh^{-1}\frac{1}{\sqrt{{\upepsilon}}})
\le{\mathit{C}}.
\label{DSf1bound}
\end{eqn}
Similarly, using \eqref{DSf02},
\begin{eqn}
\norm{{\mathpzc{f}}^{{\upepsilon}}}_{-{\frac{1}{2}}}^2
&=(\frac{{\upepsilon}}{2\pi})(\int_{|{\upxi}|>\sqrt{{\upepsilon}}}|\varpi_{{\upepsilon}}^{-{\frac{1}{2}}}{\mathpzc{f}}_{{\upepsilon}}^F|^2({\upxi})d{\upxi}+\int_{-\sqrt{{\upepsilon}}}^{+\sqrt{{\upepsilon}}}|\varpi_{{\upepsilon}}^{-{\frac{1}{2}}}{\mathpzc{f}}_{{\upepsilon}}^F|^2({\upxi})d{\upxi})\\
&\le {\upepsilon}({\mathit{C}}_1\frac{1}{{{\upepsilon}}}+{\mathit{C}}_2\int_{-1/\sqrt{{\upepsilon}}}^{+1/\sqrt{{\upepsilon}}}\frac{1}{(1+{\xi}^2)\sqrt{1+{\xi}^2}}d{\xi})
\le{\mathit{C}}.
\label{DSf2bound}
\end{eqn}
\qed
\end{proof}

In the context of the promised theme of this paper, in order to move further, an identification is required between the  continuous and discrete formulation. This is accomplished by a suitable definition of prolongation and restriction operators {\citep{Stevenson}}. 
 \begin{remark}
Since the elements of the associated function spaces for the continuum model are, typically, measurable functions, the point wise values are not directly meaningful. However, using a dense subset the notion can be worked out. The details are omitted as they can be found in standard texts on Sobolev spaces \citep{Adams}. 
\end{remark}
Let 
\begin{eqn}
\mathcal{R}_{{\upepsilon}}: {\mathscr{H}^{\mp{\frac{1}{2}}}}({\mathbb{R}})\to {\mathscr{H}^{\mp{\frac{1}{2}}}}
\label{Rlen}
\end{eqn}
denote the restriction operator. 
{The restriction operator is the same as that defined by \cite{Stevenson} (cf. definitions 3.11 and 3.13). There are some obvious and straightforward instances of the same for smooth functions.}
For example, in the context of the expression provided in \eqref{cfx}, $(\mathcal{R}_{{\upepsilon}}{\mathpzc{f}})_{{\mathtt{x}}}={\mathpzc{f}}({{\mathtt{x}}}{\upepsilon}), {{\mathtt{x}}}\in{\mathbb{Z}}^-.$ 
Further, by \eqref{dWHf},
\begin{eqn}
{\mathpzc{f}}^{{\upepsilon}}_{{\mathtt{x}}}-(\mathcal{R}_{{\upepsilon}}{\mathpzc{f}})_{{\mathtt{x}}}
&=\begin{dcases}
\frac{1}{4}(\delta_{{\mathtt{x}}+1, 0}-{\mathtt{c}}^{{\upepsilon}}_{{\mathtt{x}}+1}){{\mathtt{u}}}^{{t}}_{0, 0}+{\frac{1}{2}}(1-\cos{\upepsilon}{\mathit{k}}_y)(-{{\mathtt{u}}}^{{i}}_{{\mathtt{x}}, 0}+({\mathtt{c}}^{{\upepsilon}}\ast{{\mathtt{u}}}^{{i}}_{\cdot, 0})_{{\mathtt{x}}})&\quad\text{\rm (D)},\\
-{\frac{1}{2}}{\upepsilon}^{-1}(1+i{\upepsilon}{\mathit{k}}_y-e^{i{\upepsilon}{\mathit{k}}_y}){{\mathtt{u}}}^{{i}}_{{\mathtt{x}}, 0}+{\frac{1}{2}}{\upepsilon}^{-1}(1-e^{i{\upepsilon}{\mathit{k}}_y})({\mathtt{c}}^{{\upepsilon}}\ast{{\mathtt{u}}}^{{i}}_{\cdot, 0})_{{\mathtt{x}}}&\quad\text{\rm (N).}
\end{dcases}
\label{DSferrorbound}
\end{eqn}
In \eqref{DSferrorbound}, ${\mathtt{c}}$ is given by \eqref{dWHc}. Around ${\upxi}=0$, in the Dirichlet case of \eqref{DSferrorbound}, 
using the expression of the solution provided by \cite{Bls1}, ${{\mathtt{u}}}^{{t}}_{0, 0}\sim\sqrt{{\upepsilon}}$ as ${\upepsilon}\to0$, so that
\begin{eqn}
{\mathpzc{f}}^{{\upepsilon}}_{{\mathtt{x}}}-(\mathcal{R}_{{\upepsilon}}{\mathpzc{f}})_{{\mathtt{x}}}&\sim(1+{\xi}^2)^{1/4}(O(\frac{1}{\sqrt{{\upepsilon}}})\frac{1}{{\xi}+{\mathit{k}}}+{\frac{1}{2}}\frac{1}{i({\xi}-{\mathit{k}}_x){\upepsilon}}(-{\frac{1}{2}}{\upepsilon}^2{\mathit{k}}_y^2(1-\frac{1}{{\upepsilon}\sqrt{{\xi}^2-{\omega}^2}})))\\
&\sim O(\frac{1}{\sqrt{{\upepsilon}}}).
\label{DSf01errorbound}
\end{eqn}
Similarly, in the case of the Neumann condition of \eqref{DSferrorbound},
\begin{eqn}
{\mathpzc{f}}^{{\upepsilon}}_{{\mathtt{x}}}-(\mathcal{R}_{{\upepsilon}}{\mathpzc{f}})_{{\mathtt{x}}}\sim-(1+{\xi}^2)^{-1/4}({\upepsilon}^{-1}(-1)i{\upepsilon}{\mathit{k}}_y)(1-{\frac{1}{2}}{\upepsilon}\sqrt{{\xi}^2-{\omega}^2}-1){\frac{1}{2}}\frac{1}{i({\xi}-{\mathit{k}}_x){\upepsilon}}\sim O(1).
\label{DSf02errorbound}
\end{eqn}
Using the estimates in the manner analogous to that employed for $\norm{{\mathpzc{f}}^{{\upepsilon}}}_{\pm{\frac{1}{2}}}$ in \eqref{DSf1bound} and \eqref{DSf2bound}, the following lemma follows from \eqref{DSf01errorbound} and \eqref{DSf02errorbound}.
\begin{lemma}\begin{eqn}
\norm{{\mathpzc{f}}^{{\upepsilon}}-\mathcal{R}_{{\upepsilon}}{\mathpzc{f}}}^2_{\pm{\frac{1}{2}}}<{\mathit{C}}\begin{dcases}\sqrt{{\upepsilon}}&\text{\rm (D)},\\{\upepsilon}&\text{\rm \rm (N),}\end{dcases}
\label{ferrorbound}
\end{eqn}
where ${\mathpzc{f}}^{{\upepsilon}}$ is given by \eqref{dWHf} and ${\mathpzc{f}}$ by \eqref{cfx}.
\label{lemma4} 
\end{lemma} 
\begin{remark}
In the context of \eqref{Rlen}, the question answered in this paper is that ${\mathpzc{x}}^{{\upepsilon}}$, ($=\{{\mathpzc{x}}^{{\upepsilon}}_{{\mathtt{x}}}\}_{{\mathtt{x}}=-1}^{-\infty}$) that satisfies \eqref{dWH}, provides an approximation of $\mathcal{R}_{{\upepsilon}}{\mathpzc{x}}$ where ${\mathpzc{x}}$ satisfies \eqref{cWH2}. 
\end{remark}
Let
\begin{eqn}
\mathcal{P}_{{\upepsilon}}: {\mathscr{H}^{\mp{\frac{1}{2}}}}\to {\mathscr{H}^{\mp{\frac{1}{2}}}}({\mathbb{R}})
\label{Plen}
\end{eqn}
stand for a suitable prolongation operator, i.e. an operator which is defined in a direction opposite to that of $\mathcal{R}_{{\upepsilon}}$ such that their composition satisfies an approximate identity map.

The existence and uniqueness of the solution of the discrete Sommerfeld problems in discrete Sobolev space, analogous to the continuous case, follows from the results stated by \cite{Hackbusch}. 
This is due to the following two facts.
\begin{enumerate}
\item The existence and uniqueness of the solution in $\ell_2$ of the discrete Wiener--Hopf equation, when ${\omega}_2>0$, for the discrete Sommerfeld problems, is established by \cite{Bls2, Bls3} 
where it is shown that the 
relevant Toeplitz operator is invertible on $\ell_2$ using the results of \cite{Krein, Widom1} and verification of the Krein conditions.
\item For any $s\in{\mathbb{R}}$ and positive ${\upepsilon}$, the discrete Sobolev space ${\mathscr{H}^s}$ is equivalent to $\ell_2({\mathbb{Z}}).$
\end{enumerate}
Hence, analogous to \eqref{cKinv},
\begin{eqn}
&\text{for any positive }{\upepsilon}, \text{the operator } \mathcal{K}^{{\upepsilon}}: {\mathscr{H}^{\mp{\frac{1}{2}}}}\to{\mathscr{H}^{\pm{\frac{1}{2}}}}, \text{ defined by }\eqref{dWH} \text{ and }\eqref{dWHfull},\\
&\text{ is bijective and continuous for each of the two discrete Sommerfeld problems.}
\label{dKinv}
\end{eqn}
Thus, by \eqref{cKinv}, $\norm{\mathcal{K}}<\infty$ and \eqref{dKinv} implies $\norm{\mathcal{K}^{{\upepsilon}}}<\infty$ where $\norm{\cdot}$ refers to the corresponding operator norm. 

\begin{remark}
Before going over to the main result stated and proved in the next section, it is pertinent to recall \citep{Lax} that a sesquilinear form (on a Hilbert space $H$) is called bounded iff there is a $C > 0$ such that
\begin{eqn}
|a(x, y)| \le C\norm{x}\norm{y}, x, y \in H.
\label{Boundcond}
\end{eqn}
Also recall that sesquilinear form is called $H$-elliptic (or coercive) iff there is a $c > 0$ such that
\begin{eqn}
\Re a(x, x) \ge c\norm{x}^2, x \in H.
\label{Coerccond}
\end{eqn}
\end{remark}
For any $u_{{\upepsilon}}$, with a fixed choice of $s={\mp{\frac{1}{2}}},$ the pairing $\langle\mathcal{K}^{{\upepsilon}}u_{{\upepsilon}}, v_{{\upepsilon}}\rangle$ defines a bounded sesquilinear form on ${\mathscr{H}^{s}}$ (for every $v_{{\upepsilon}}\in{\mathscr{H}^{s}}$). 
But,
\begin{eqn}
\Re\langle\mathcal{K}^{{\upepsilon}}v_{{\upepsilon}}, v_{{\upepsilon}}\rangle&=\Re{\upepsilon}\int_I (\mathcal{K}^{{\upepsilon}}v_{{\upepsilon}})^F({\upxi})\overline{v_{{\upepsilon}}^F({\upxi})}d{\upxi}\\
&\ge\frac{{\mathit{C}}}{\alpha}{\upepsilon}\int_I \varpi_{{\upepsilon}}^{2s}({\upxi})v_{{\upepsilon}}^F({\upxi})\overline{v_{{\upepsilon}}^F({\upxi})}d{\upxi}=\frac{{\mathit{C}}}{\alpha}\norm{v_{{\upepsilon}}}^2_s,
\label{CoercK}
\text{where }
\Re({\mathtt{k}}^{{\upepsilon}})^F({\upxi})\ge\alpha, {\upxi}\in I.
\end{eqn}
In fact, \eqref{branch} implies that $\alpha={\omega}_2^{\pm1}$ (independent of ${\upepsilon}$). Thus, by the Lax--Milgram theorem \citep{Lax}, $\mathcal{K}^{{\upepsilon}}$ is invertible in the limit ${\upepsilon}\to0^+$. 
{The provided statement, thus, corroborates the statement \eqref{dKinv} and also gives an explicit bound which can be used in the proof of the main result.}
The next section presents a proof of the main theorem in this paper. A numerical illustration also appears in the end.

\section{Main result} 
\begin{theorem}
For ${\omega}_2>0$, 
\begin{eqn}
\norm{{\mathpzc{x}}^{{\upepsilon}}-\mathcal{R}_{{\upepsilon}}{\mathpzc{x}}}_{\mp{\frac{1}{2}}}\to0\text{ as }{\upepsilon}\to0,
\label{maineqn}
\end{eqn}
where ${\mathpzc{x}}^{{\upepsilon}}$ is the solution of 
\eqref{dWH} and ${\mathpzc{x}}$ is the solution of 
\eqref{cWH2}. 
\label{uDCthm} 
\end{theorem} 
\begin{remark}
Again, recall that, according to the notation stated in the beginning of the paper, the upper symbol in $\mp$ refers to Dirichlet (D) and lower to Neumann (N).
\end{remark}

\begin{remark}
A partial statement of Theorem above is stated as Theorem 4.4 of \citep{Bls2} corresponding to the part involving $(N)$.
\end{remark}

\begin{remark}
At this point, it is noteworthy that the statement 
without the assumption of dissipation, i.e., ${\upomega}_2=0$, is also an interesting aspect of the problem, but unresolved by the author as yet (partly because it has not been possible to tackle invertibility for the discrete problem when ${\upomega}_2=0$ \citep{Bls2, Bls3}). 
\end{remark}

\begin{proof}
Since ${\mathcal{K}^{{\upepsilon}}}$ is invertible, 
\begin{eqn}
{\mathpzc{x}}^{{\upepsilon}}-\mathcal{R}_{{\upepsilon}}{\mathpzc{x}}&={\mathcal{K}^{{\upepsilon}}}^{-1}({\mathpzc{f}}^{{\upepsilon}}-{\mathcal{R}_{{\upepsilon}}}{\mathpzc{f}})+{\mathcal{K}^{{\upepsilon}}}^{-1}({\mathcal{R}_{{\upepsilon}}}{\mathpzc{f}}-{\mathcal{K}^{{\upepsilon}}}\mathcal{R}_{{\upepsilon}}{\mathpzc{x}}).
\label{eq1}
\end{eqn}
To proceed further, since ${\mathcal{K}^{{\upepsilon}}}^{-1}$ is bounded, 
the Lemma \ref{lemma4} can be applied to estimate the first term, hence, mainly the second term in \eqref{eq1} (inside brackets) needs to be estimated. 

Suppose that 
\begin{eqn}
{\mathpzc{k}}={\mathpzc{k}}^{ns}+{\mathpzc{k}}^{sing}
\label{eq11}
\end{eqn}
is the decomposition of ${\mathpzc{k}}$ into non-singular and singular part such that the support of ${\mathpzc{k}}^{sing}$ is $S$ (which can be made as small as desired). Note that $0\in S$ since ${\mathpzc{k}}({\mathit{x}})$ is singular kernel at ${\mathit{x}}=0$. Indeed, separating the non-singular and singular terms, a rearrangement of the second term in bracket in \eqref{eq1} yields
\begin{eqn}
&{\mathcal{R}_{{\upepsilon}}}{\mathpzc{f}}_{{\mathtt{x}}}-({\mathcal{K}^{{\upepsilon}}}\mathcal{R}_{{\upepsilon}}{\mathpzc{x}})_{{\mathtt{x}}}\\
&={\mathpzc{f}}({{\mathtt{x}}}{\upepsilon})-({\mathcal{K}^{{\upepsilon}}}\mathcal{R}_{{\upepsilon}}{\mathpzc{x}})_{{\mathtt{x}}}\\&=(\mathcal{K}{\mathpzc{x}})({{\mathtt{x}}}{\upepsilon})-({\mathcal{K}^{{\upepsilon}}}\mathcal{R}_{{\upepsilon}}{\mathpzc{x}})_{{\mathtt{x}}}=({\mathpzc{k}}\ast{\mathpzc{x}})({{\mathtt{x}}}{\upepsilon})-({\mathtt{k}}^{{\upepsilon}}\ast\mathcal{R}_{{\upepsilon}}{\mathpzc{x}})_{{\mathtt{x}}}\\&=\overset{{\mathit{C}}_{{\mathtt{x}}}{\upepsilon}}{\overbrace{{\upepsilon}\int_{-{\frac{1}{2}}}^{0}{\mathpzc{k}}({\mathtt{x}}{\upepsilon}-t{\upepsilon}){\mathpzc{x}}(t{\upepsilon})dt}}+(({\mathpzc{k}}^{ns}+{\mathpzc{k}}^{sing})\ast{\mathpzc{x}})({{\mathtt{x}}}{\upepsilon})-({\mathtt{k}}^{{\upepsilon}}\ast\mathcal{R}_{{\upepsilon}}{\mathpzc{x}})_{{\mathtt{x}}}\\&={\mathit{C}}_{{\mathtt{x}}}{\upepsilon}+(\mathcal{R}_{{\upepsilon}}({\mathpzc{k}}^{ns}\ast{\mathpzc{x}}))_{{\mathtt{x}}}-({\mathtt{k}}^{{\upepsilon}}\ast\mathcal{R}_{{\upepsilon}}{\mathpzc{x}})_{{\mathtt{x}}}+({\mathpzc{k}}^{sing}\ast{\mathpzc{x}})({{\mathtt{x}}}{\upepsilon})\\&={\mathit{C}}_{{\mathtt{x}}}{\upepsilon}+(\mathcal{R}_{{\upepsilon}}({\mathpzc{k}}^{ns}\ast({\mathpzc{x}}-\mathcal{P}_{{\upepsilon}}\mathcal{R}_{{\upepsilon}}{\mathpzc{x}}+\mathcal{P}_{{\upepsilon}}\mathcal{R}_{{\upepsilon}}{\mathpzc{x}})))_{{\mathtt{x}}}-({\mathtt{k}}^{{\upepsilon}}\ast\mathcal{R}_{{\upepsilon}}{\mathpzc{x}})_{{\mathtt{x}}}\\
&+({\mathpzc{k}}^{sing}\ast{\mathpzc{x}})({{\mathtt{x}}}{\upepsilon})\\&={\mathit{C}}_{{\mathtt{x}}}{\upepsilon}+(\mathcal{R}_{{\upepsilon}}({\mathpzc{k}}^{ns}\ast({\mathpzc{x}}-\mathcal{P}_{{\upepsilon}}\mathcal{R}_{{\upepsilon}}{\mathpzc{x}})))_{{\mathtt{x}}}+(\mathcal{R}_{{\upepsilon}}({\mathpzc{k}}^{ns}\ast\mathcal{P}_{{\upepsilon}}\mathcal{R}_{{\upepsilon}}{\mathpzc{x}}))_{{\mathtt{x}}}-({\mathtt{k}}^{{\upepsilon}}\ast\mathcal{R}_{{\upepsilon}}{\mathpzc{x}})_{{\mathtt{x}}}\\
&+({\mathpzc{k}}^{sing}\ast{\mathpzc{x}})({{\mathtt{x}}}{\upepsilon})\\&=(\mathcal{R}_{{\upepsilon}}\mathscr{I})_{{\mathtt{x}}}+(\mathcal{R}_{{\upepsilon}}({\mathpzc{k}}^{ns}\ast({\mathpzc{x}}-\mathcal{P}_{{\upepsilon}}\mathcal{R}_{{\upepsilon}}{\mathpzc{x}})))_{{\mathtt{x}}}\\&+(\mathcal{R}_{{\upepsilon}}({\mathpzc{k}}^{ns}\ast\mathcal{P}_{{\upepsilon}}\mathcal{R}_{{\upepsilon}}{\mathpzc{x}}))_{{\mathtt{x}}}-({\mathtt{k}}^{{\upepsilon}}\ast\mathcal{R}_{{\upepsilon}}{\mathpzc{x}})_{{\mathtt{x}}}+({\mathpzc{k}}^{sing}\ast{\mathpzc{x}})({{\mathtt{x}}}{\upepsilon}).
\label{eq2t}
\end{eqn}
Equivalently, above can be expressed as
\begin{eqn}
{\mathcal{R}_{{\upepsilon}}}{\mathpzc{f}}-{\mathcal{K}^{{\upepsilon}}}\mathcal{R}_{{\upepsilon}}{\mathpzc{x}}&=\mathcal{R}_{{\upepsilon}}\mathscr{I}+\mathcal{R}_{{\upepsilon}}({\mathpzc{k}}^{ns}\ast({\mathpzc{x}}-\mathcal{P}_{{\upepsilon}}\mathcal{R}_{{\upepsilon}}{\mathpzc{x}}))\\
&+\mathcal{R}_{{\upepsilon}}({\mathpzc{k}}^{ns}\ast\mathcal{P}_{{\upepsilon}}\mathcal{R}_{{\upepsilon}}{\mathpzc{x}})-{\mathtt{k}}^{{\upepsilon}}\ast\mathcal{R}_{{\upepsilon}}{\mathpzc{x}}+\mathcal{R}_{{\upepsilon}}({\mathpzc{k}}^{sing}\ast{\mathpzc{x}}),
\label{eq2}
\end{eqn}
where 
\begin{eqn}
\mathscr{I}({\mathit{x}})={\upepsilon}\int_{-{\frac{1}{2}}}^{0}{\mathpzc{k}}({\mathit{x}}-t{\upepsilon}){\mathpzc{x}}(t{\upepsilon})dt, {\mathit{x}}\in{\mathbb{R}}^-.
\label{eq21}
\end{eqn}
Since $\mathcal{K}$ is a bounded operator on $\mathscr{H}^{\mp{\frac{1}{2}}}({\mathbb{R}}^-)$ as stated in \eqref{cKinv}, ${\mathpzc{k}}^{ns}\ast({\mathpzc{x}}-\mathcal{P}_{{\upepsilon}}\mathcal{R}_{{\upepsilon}}{\mathpzc{x}})$ is bounded by the operator norm of $\mathcal{K}$ and the norm of ${\mathpzc{x}}-\mathcal{P}_{{\upepsilon}}\mathcal{R}_{{\upepsilon}}{\mathpzc{x}}$ in $\mathscr{H}^{\mp{\frac{1}{2}}}({\mathbb{R}}^-)$. The latter goes to zero as ${\upepsilon}\to0$ after the prolongation operator $\mathcal{P}_{{\upepsilon}}$ defined by \eqref{Plen} is suitably chosen such that $\mathcal{P}_{{\upepsilon}}\mathcal{R}_{{\upepsilon}}$ is an approximation of the identity operator $\mathrm{id}$ on $\mathscr{H}^{\mp{\frac{1}{2}}}({\mathbb{R}}^-)$ (see (4.3) stated by \cite{Hackbusch}
). Thus, the remaining `error' estimate in \eqref{eq2} concerns the last three terms in \eqref{eq2} 
which can be rearranged and expressed as 
\begin{eqn}
&(\mathcal{R}_{{\upepsilon}}({\mathpzc{k}}^{ns}\ast\mathcal{P}_{{\upepsilon}}\mathcal{R}_{{\upepsilon}}{\mathpzc{x}}))_{{\mathtt{x}}}-({\mathtt{k}}^{{\upepsilon}}\ast\mathcal{R}_{{\upepsilon}}{\mathpzc{x}})_{{\mathtt{x}}}+({\mathpzc{k}}^{sing}\ast{\mathpzc{x}})({{\mathtt{x}}}{\upepsilon})\\
&=\int_{-\infty}^{-{\frac{1}{2}}{\upepsilon}}{\mathpzc{k}}^{ns}({\mathtt{x}}{\upepsilon}-t)\mathcal{P}_{{\upepsilon}}\mathcal{R}_{{\upepsilon}}{\mathpzc{x}}({t})dt-\sum_{j=-\infty}^{-1}{\mathtt{k}}^{{\upepsilon}}_{{\mathtt{x}}-j}(\mathcal{R}_{{\upepsilon}}{\mathpzc{x}})_{j}+({\mathpzc{k}}^{sing}\ast{\mathpzc{x}})({{\mathtt{x}}}{\upepsilon})\\&=\int_{-\infty}^{-{\frac{1}{2}}{\upepsilon}}{\mathpzc{k}}^{ns}({\mathtt{x}}{\upepsilon}-t)(\mathcal{P}_{{\upepsilon}}\mathcal{R}_{{\upepsilon}}{\mathpzc{x}}({t})-{\mathpzc{x}}(\lfloor{t}/{\upepsilon}\rfloor{\upepsilon})+{\mathpzc{x}}(\lfloor{t}/{\upepsilon}\rfloor{\upepsilon}))dt-\sum_{j=-\infty}^{-1}{\mathtt{k}}^{{\upepsilon}}_{{\mathtt{x}}-j}(\mathcal{R}_{{\upepsilon}}{\mathpzc{x}})_{j}\\
&+({\mathpzc{k}}^{sing}\ast{\mathpzc{x}})({{\mathtt{x}}}{\upepsilon})\\&=\int_{-\infty}^{-{\frac{1}{2}}{\upepsilon}}{\mathpzc{k}}^{ns}({\mathtt{x}}{\upepsilon}-t)(\mathcal{P}_{{\upepsilon}}\mathcal{R}_{{\upepsilon}}{\mathpzc{x}}({t})-{\mathpzc{x}}(\lfloor{t}/{\upepsilon}\rfloor{\upepsilon}))dt\\&+\int_{-\infty}^{{-{\frac{1}{2}}{\upepsilon}}}{\mathpzc{k}}^{ns}({\mathtt{x}}{\upepsilon}-t){\mathpzc{x}}(\lfloor{t}/{\upepsilon}\rfloor{\upepsilon})dt-\sum_{j=-\infty}^{-1}{\mathtt{k}}^{{\upepsilon}}_{{\mathtt{x}}-j}{\mathpzc{x}}({j}{\upepsilon})+({\mathpzc{k}}^{sing}\ast{\mathpzc{x}})({{\mathtt{x}}}{\upepsilon})\\&=\int_{-\infty}^{-{\frac{1}{2}}{\upepsilon}}{\mathpzc{k}}^{ns}({\mathtt{x}}{\upepsilon}-t)(\mathcal{P}_{{\upepsilon}}\mathcal{R}_{{\upepsilon}}{\mathpzc{x}}({t})-{\mathpzc{x}}(\lfloor{t}/{\upepsilon}\rfloor{\upepsilon}))dt\\&+\sum_{j=-\infty}^{-1}\int_{-{\frac{1}{2}}}^{+{\frac{1}{2}}}{\mathpzc{k}}^{ns}({\mathtt{x}}{\upepsilon}-j{\upepsilon}-t{\upepsilon}){\mathpzc{x}}(\lfloor{j{\upepsilon}}/{\upepsilon}\rfloor{\upepsilon})dt-\sum_{j=-\infty}^{-1}{\mathtt{k}}^{{\upepsilon}}_{{\mathtt{x}}-j}{\mathpzc{x}}({j}{\upepsilon})+({\mathpzc{k}}^{sing}\ast{\mathpzc{x}})({{\mathtt{x}}}{\upepsilon})\\&=\int_{-\infty}^{-{\frac{1}{2}}{\upepsilon}}{\mathpzc{k}}^{ns}({\mathtt{x}}{\upepsilon}-t)(\mathcal{P}_{{\upepsilon}}\mathcal{R}_{{\upepsilon}}{\mathpzc{x}}({t})-{\mathpzc{x}}(\lfloor{t}/{\upepsilon}\rfloor{\upepsilon}))dt\\&+\sum_{j=-\infty}^{-1}\int_{-{\frac{1}{2}}}^{+{\frac{1}{2}}}{\mathpzc{k}}^{ns}({\mathtt{x}}{\upepsilon}-j{\upepsilon}-t{\upepsilon}){\mathpzc{x}}(j{\upepsilon})dt-\sum_{j=-\infty}^{-1}{\mathtt{k}}^{{\upepsilon}}_{{\mathtt{x}}-j}{\mathpzc{x}}({j}{\upepsilon})+({\mathpzc{k}}^{sing}\ast{\mathpzc{x}})({{\mathtt{x}}}{\upepsilon})\\&=\int_{-\infty}^{-{\frac{1}{2}}{\upepsilon}}{\mathpzc{k}}^{ns}({\mathtt{x}}{\upepsilon}-t)(\mathcal{P}_{{\upepsilon}}\mathcal{R}_{{\upepsilon}}{\mathpzc{x}}({t})-{\mathpzc{x}}(\lfloor{t}/{\upepsilon}\rfloor{\upepsilon}))dt\\&+\sum_{j=-\infty}^{-1}(\overset{\overline{{\mathtt{k}}}_{{\mathtt{x}}-j}^{{ns}:{\upepsilon}}}{\overbrace{{\upepsilon}\int_{-{\frac{1}{2}}}^{+{\frac{1}{2}}}{\mathpzc{k}}^{ns}({\mathtt{x}}{\upepsilon}-j{\upepsilon}-t{\upepsilon})dt}}){\mathpzc{x}}(j{\upepsilon})-\sum_{j=-\infty}^{-1}{\mathtt{k}}^{{\upepsilon}}_{{\mathtt{x}}-j}{\mathpzc{x}}({j}{\upepsilon})+({\mathpzc{k}}^{sing}\ast{\mathpzc{x}})({{\mathtt{x}}}{\upepsilon})\\
&=\bigg(\sum_{j=-\infty}^{-1}({\upepsilon}\int_{-{\frac{1}{2}}}^{+{\frac{1}{2}}}{\mathpzc{k}}^{ns}({\mathtt{x}}{\upepsilon}-j{\upepsilon}-t{\upepsilon})(\mathcal{P}_{{\upepsilon}}\mathcal{R}_{{\upepsilon}}{\mathpzc{x}}({t}{\upepsilon}))dt-\overline{{\mathtt{k}}}_{{\mathtt{x}}-j}^{{ns}:{\upepsilon}}{\mathpzc{x}}({j}{\upepsilon}))\bigg)\\
&+\sum_{j=-\infty}^{-1}(\overline{{\mathtt{k}}}_{{\mathtt{x}}-j}^{{ns}:{\upepsilon}}-\widetilde{{\mathtt{k}}}^{{\upepsilon}}_{{\mathtt{x}}-j}){\mathpzc{x}}({j}{\upepsilon})+\bigg(-\sum_{j\in S^{{\upepsilon}}}{{\mathtt{k}}}^{{\upepsilon}}_{{\mathtt{x}}-j}{\mathpzc{x}}({j}{\upepsilon})+({\mathpzc{k}}^{sing}\ast{\mathpzc{x}})({{\mathtt{x}}}{\upepsilon})\bigg),
\label{eq3}
\end{eqn}
where 
\begin{eqn}
\overline{{\mathtt{k}}}_{{\mathtt{x}}-j}^{{ns}:{\upepsilon}}={\upepsilon}\int_{-{\frac{1}{2}}}^{+{\frac{1}{2}}}{\mathpzc{k}}^{ns}({\mathtt{x}}{\upepsilon}-j{\upepsilon}-t{\upepsilon})dt,
\label{eq31}
\end{eqn}
and $\widetilde{{\mathtt{k}}}^{{\upepsilon}}$ is defined such that $\widetilde{{\mathtt{k}}}^{{\upepsilon}}_{{\mathtt{x}}-j}={\mathtt{k}}^{{\upepsilon}}_{{\mathtt{x}}-j}$ outside the set $S^{{\upepsilon}}=S\cap({\upepsilon}{\mathbb{Z}})$ and equal to $0$ inside it. 
The first term in \eqref{eq3} can be controlled by ${\upepsilon}$ due to the presence of $\mathcal{P}_{{\upepsilon}}\mathcal{R}_{{\upepsilon}}$ as an approximation of the identity operator. For the second term in \eqref{eq3}, it is clear that there is a standard estimate (Riemann sum) for $\overline{{\mathtt{k}}}_{{\mathtt{x}}-j}^{{ns}:{\upepsilon}}-\widetilde{{\mathtt{k}}}^{{\upepsilon}}_{{\mathtt{x}}-j}$ since the kernel ${\mathpzc{k}}^{ns}(t)$ is non-singular. For $j\ne{\mathtt{x}}$, there is also an estimate for the singular part ${\mathpzc{k}}^{sing}(t)$ at $t\ne0$ while at $t=0$, ${\mathpzc{k}}^{sing}(t)$ is logarithmic for (D) and hypersingular for (N) implying that more care is required for its estimate. 

Outside the finite set $S^{{\upepsilon}}$, using the far field expansion of Hankel function as well as the square Lattice Green's function, it can be easily shown that \begin{eqn}
{\mathtt{k}}^{{\upepsilon}}_{{\mathtt{x}}}-{\mathpzc{k}}({\mathtt{x}}{\upepsilon}) {\upepsilon}={\upepsilon} O(\frac{e^{-{\upkappa}_2 |{{\mathtt{x}}}|}}{|{\mathtt{x}}{\upkappa}|})={\upepsilon} O(\frac{e^{-{\mathit{k}}_2 |{{\mathit{x}}}|}}{|{\mathit{x}}{\mathit{k}}|}), 
\label{}
\end{eqn}
so that by an application of mean value theorem to \eqref{eq31}, the second term in \eqref{eq3} can be controlled by ${\upepsilon}$. 
Finally, the third, and the last term, that needs to be estimated in \eqref{eq3} is the singular part $-\sum_{j\in S^{{\upepsilon}}}{{\mathtt{k}}}^{{\upepsilon}}_{{\mathtt{x}}-j}{\mathpzc{x}}({j}{\upepsilon})+({\mathpzc{k}}^{sing}\ast{\mathpzc{x}})({{\mathtt{x}}}{\upepsilon})$. For the Dirichlet boundary condition, the kernel is logarithmic and therefore, using the definition of $\overline{{\mathtt{k}}}_{0}^{{sing}:{\upepsilon}}$ analogous to \eqref{eq31}, 
\begin{eqn}
|\overline{{\mathtt{k}}}_{0}^{{sing}:{\upepsilon}}|=|{\upepsilon}\int_{-{\frac{1}{2}}}^{+{\frac{1}{2}}}{\mathpzc{k}}(-t{\upepsilon})dt|<{\mathit{C}}{\upepsilon}\log{\upepsilon},
\label{}
\end{eqn}
which yields $\norm{\sum_{j\in S^{{\upepsilon}}}(\overline{{\mathtt{k}}}_{{\mathtt{x}}-j}^{{sing}:{\upepsilon}}-{{\mathtt{k}}}_{{\mathtt{x}}-j}){\mathpzc{x}}({j}{\upepsilon})}_{\ell_2}=O(1)$. For the Neumann boundary condition, the kernel is hypersingular  and, therefore, a coarse estimate is 
\begin{eqn}
\overline{{\mathtt{k}}}_{0}^{{sing}:{\upepsilon}}={\upepsilon}\int_{-{\frac{1}{2}}}^{+{\frac{1}{2}}}{\mathpzc{k}}(-t{\upepsilon})dt=-\frac{2}{\pi} \frac{1}{{\upepsilon}},
\label{}
\end{eqn}
which gives $\norm{\sum_{j\in S^{{\upepsilon}}}(\overline{{\mathtt{k}}}_{{\mathtt{x}}-j}^{{sing}:{\upepsilon}}-{{\mathtt{k}}}_{{\mathtt{x}}-j}){\mathpzc{x}}({j}{\upepsilon})}_{\ell_2}=O({\upepsilon}^{-1})$ and, undoubtedly, it is unacceptable. But a refinement is possible. 

For the Neumann boundary condition, let $\beta\in(0, 1)$. Then the last term in bracket in \eqref{eq3} can be simplified with a rearrangement, using a definition of $\overline{{\mathtt{k}}}_{0}^{{sing}:{\upepsilon}}$ analogous to \eqref{eq31},
\begin{eqn}
&-\sum_{j\in S^{{\upepsilon}}}{{\mathtt{k}}}^{{\upepsilon}}_{{\mathtt{x}}-j}{\mathpzc{x}}({j}{\upepsilon})+({\mathpzc{k}}^{sing}\ast{\mathpzc{x}})({{\mathtt{x}}}{\upepsilon})\\
&=\sum_{j\in S^{{\upepsilon}}}(\beta
\overline{{\mathtt{k}}}_{{\mathtt{x}}-j}^{{sing}:{\upepsilon}}
-{{\mathtt{k}}}^{{\upepsilon}}_{{\mathtt{x}}-j}){\mathpzc{x}}({j}{\upepsilon})\\
&+\beta\bigg(\sum_{j=-\infty}^{-1}({\upepsilon}\int_{-{\frac{1}{2}}}^{+{\frac{1}{2}}}{\mathpzc{k}}^{sing}({\mathtt{x}}{\upepsilon}-j{\upepsilon}-t{\upepsilon})(\mathcal{P}_{{\upepsilon}}\mathcal{R}_{{\upepsilon}}{\mathpzc{x}}({t}{\upepsilon}))dt-\overline{{\mathtt{k}}}_{{\mathtt{x}}-j}^{{sing}:{\upepsilon}}{\mathpzc{x}}({j}{\upepsilon}))\bigg)\\
&+\beta\bigg({\mathpzc{k}}^{sing}\ast({\mathpzc{x}}-\mathcal{P}_{{\upepsilon}}\mathcal{R}_{{\upepsilon}}{\mathpzc{x}})({{\mathtt{x}}}{\upepsilon})\bigg)+(1-\beta)({\mathpzc{k}}^{sing}\ast{\mathpzc{x}})({{\mathtt{x}}}{\upepsilon}).
\label{eq5}
\end{eqn}
Since ${{\mathtt{k}}}^{{\upepsilon}}_{0}=-{\upepsilon}^{-1}(1-2({{\mathcal{G}}}^{{\upepsilon}}_{0, 1}-{{\mathcal{G}}}^{{\upepsilon}}_{0, 0}))=-{\upepsilon}^{-1}{\frac{1}{2}}(1+{\upepsilon}^2{\omega}^2{{\mathcal{G}}}^{{\upepsilon}}_{0, 0})$ as ${\upepsilon}\to0$, $\beta$ is chosen to be $\frac{\pi}{4}$ and the first term in \eqref{eq5} is controlled. By the same reason as for ${\mathpzc{k}}^{ns}$ provided earlier, it follows that the contribution of second and third terms (in brackets)  vanishes as ${\upepsilon}\to0$. 
Since ${\mathpzc{k}}^{sing}$ is only a part of a kernel of bounded operator, the norm of $({\mathpzc{k}}^{sing}\ast{\mathpzc{x}})$ is dominated by the measure of $S$ which can be chosen to diminish as ${\upepsilon}\to0.$ 

Collecting all terms together with re-arrangements described above, equation \eqref{eq1} can be written as
\begin{eqn}
{\mathpzc{x}}^{{\upepsilon}}_{{\mathtt{x}}}-{\mathpzc{x}}({\mathtt{x}}{\upepsilon})&={\mathcal{K}^{{\upepsilon}}}^{-1}\bigg(({\mathpzc{f}}^{{\upepsilon}}_{{\mathtt{x}}}-{\mathcal{R}_{{\upepsilon}}}{\mathpzc{f}}_{{\mathtt{x}}})+(\mathcal{R}_{{\upepsilon}}({\mathpzc{k}}^{ns}\ast({\mathpzc{x}}-\mathcal{P}_{{\upepsilon}}\mathcal{R}_{{\upepsilon}}{\mathpzc{x}})))_{{\mathtt{x}}}\\
&+(\mathcal{R}_{{\upepsilon}}\mathscr{I})_{{\mathtt{x}}}+\sum_{j=-\infty}^{-1}({\upepsilon}\int_{-{\frac{1}{2}}}^{+{\frac{1}{2}}}{\mathpzc{k}}^{ns}({\mathtt{x}}{\upepsilon}-j{\upepsilon}-t{\upepsilon})(\mathcal{P}_{{\upepsilon}}\mathcal{R}_{{\upepsilon}}{\mathpzc{x}}({t}{\upepsilon}))dt-\overline{{\mathtt{k}}}_{{\mathtt{x}}-j}^{{ns}:{\upepsilon}}{\mathpzc{x}}({j}{\upepsilon}))\\
&+\sum_{j=-\infty}^{-1}(\overline{{\mathtt{k}}}_{{\mathtt{x}}-j}^{{ns}:{\upepsilon}}-\widetilde{{\mathtt{k}}}^{{\upepsilon}}_{{\mathtt{x}}-j}){\mathpzc{x}}({j}{\upepsilon})+\sum_{j\in S^{{\upepsilon}}}(\beta\overline{{\mathtt{k}}}_{{\mathtt{x}}-j}^{{sing}:{\upepsilon}}-{{\mathtt{k}}}^{{\upepsilon}}_{{\mathtt{x}}-j}){\mathpzc{x}}({j}{\upepsilon})\\
&+\beta{\mathpzc{k}}^{sing}\ast({\mathpzc{x}}-\mathcal{R}_{{\upepsilon}}{\mathpzc{x}})({{\mathtt{x}}}{\upepsilon})+(1-\beta)({\mathpzc{k}}^{sing}\ast{\mathpzc{x}})({{\mathtt{x}}}{\upepsilon})\bigg).
\label{eq6}
\end{eqn}
Using above expression
\begin{eqn}\norm{{\mathpzc{x}}^{{\upepsilon}}-\mathcal{R}_{{\upepsilon}}{\mathpzc{x}}}_{\mp{\frac{1}{2}}}&\le\norm{{\mathcal{K}^{{\upepsilon}}}^{-1}}\norm{{\mathpzc{f}}^{{\upepsilon}}_{{\mathtt{x}}}-{\mathcal{R}_{{\upepsilon}}}{\mathpzc{f}}_{{\mathtt{x}}}}_{\pm{\frac{1}{2}}}+\norm{{\mathcal{K}^{{\upepsilon}}}^{-1}}\bigg({\mathit{C}}_1\norm{\mathcal{K}^{ns}}\norm{{\mathpzc{x}}-\mathcal{R}_{{\upepsilon}}{\mathpzc{x}}}_{\mp{\frac{1}{2}}}\\&+\norm{\mathcal{K}}\norm{{\mathpzc{x}}}_{\mp{\frac{1}{2}}}{\upepsilon}+\norm{\overline{{\mathtt{k}}}^{{ns}:{\upepsilon}}-\widetilde{{\mathtt{k}}}^{{\upepsilon}}}\norm{{\mathpzc{x}}}_{\mp{\frac{1}{2}}}+\norm{\sum_{j\in S^{{\upepsilon}}}O(1){\mathpzc{x}}({j}{\upepsilon})}_{\mp{\frac{1}{2}}}\\&+\beta\norm{\mathcal{K}^{sing}}\norm{{\mathpzc{x}}-\mathcal{R}_{{\upepsilon}}{\mathpzc{x}}}_{\mp{\frac{1}{2}}}+(1-\beta)\norm{\mathcal{K}^{sing}}\norm{{\mathpzc{x}}}_{\mp{\frac{1}{2}}}\bigg)\\&\le{\mathit{C}}\bigg({\upepsilon}^{{\frac{1}{2}}}+\norm{\id{}-\mathcal{P}_{{\upepsilon}}\mathcal{R}_{{\upepsilon}}}+{\upepsilon}+\norm{\overline{{\mathtt{k}}}^{{ns}:{\upepsilon}}-\widetilde{{\mathtt{k}}}^{{\upepsilon}}}+\text{meas}(S^{{\upepsilon}})+\norm{\id{}-\mathcal{P}_{{\upepsilon}}\mathcal{R}_{{\upepsilon}}}\\&+\norm{\mathcal{K}^{sing}}\text{meas}(S^{{\upepsilon}})\bigg)\to0\text{ as }{\upepsilon}\to0,\label{eq7}\end{eqn}
but
\begin{eqn}
&\norm{{\mathpzc{f}}^{{\upepsilon}}-{\mathcal{R}_{{\upepsilon}}}{\mathpzc{f}}}_{\pm{\frac{1}{2}}}+\bigg({\mathit{C}}_1\norm{\mathcal{K}^{ns}}\norm{{\mathpzc{x}}-\mathcal{R}_{{\upepsilon}}{\mathpzc{x}}}_{\mp{\frac{1}{2}}}+\norm{\mathcal{K}}\norm{{\mathpzc{x}}}_{\mp{\frac{1}{2}}}{\upepsilon}+\norm{\overline{{\mathtt{k}}}^{{ns}:{\upepsilon}}-\widetilde{{\mathtt{k}}}^{{\upepsilon}}}\norm{{\mathpzc{x}}}_{\mp{\frac{1}{2}}}\\
&+\norm{\sum_{j\in S^{{\upepsilon}}}O(1){\mathpzc{x}}({j}{\upepsilon})}_{\mp{\frac{1}{2}}}+\beta\norm{\mathcal{K}^{sing}}\norm{{\mathpzc{x}}-\mathcal{R}_{{\upepsilon}}{\mathpzc{x}}}_{\mp{\frac{1}{2}}}+(1-\beta)\norm{\mathcal{K}^{sing}}\norm{{\mathpzc{x}}}_{\mp{\frac{1}{2}}}\bigg)\\
&\le\norm{{\mathpzc{f}}^{{\upepsilon}}_{{\mathtt{x}}}-{\mathcal{R}_{{\upepsilon}}}{\mathpzc{f}}_{{\mathtt{x}}}}_{\pm{\frac{1}{2}}}
+{\mathit{C}}\bigg(\norm{\id{}-\mathcal{P}_{{\upepsilon}}\mathcal{R}_{{\upepsilon}}}+{\upepsilon}+\norm{\overline{{\mathtt{k}}}^{{ns}:{\upepsilon}}-\widetilde{{\mathtt{k}}}^{{\upepsilon}}}+\text{meas}(S)+\norm{\id{}-\mathcal{P}_{{\upepsilon}}\mathcal{R}_{{\upepsilon}}}\\
&+\norm{\mathcal{K}^{sing}}\text{meas}(S)\bigg)\to0\text{ as }{\upepsilon}\to0.
\label{eq72}
\end{eqn}
Note that $\norm{\cdot}$ refers to the corresponding operator norm. The projection-restriction operators are chosen (with $\mathrm{id}$ as an identity operator) such that $\norm{\mathrm{id}-\mathcal{P}_{{\upepsilon}}\mathcal{R}_{{\upepsilon}}}$ can be made as small as desired \citep{Hackbusch}, as well as the Lebesgue measure of singular set $S$. Using Lemma \ref{lemma4} and arguments presented above, from \eqref{eq7} it can be concluded that ${\mathcal{K}^{{\upepsilon}}}({\mathpzc{x}}^{{\upepsilon}}_{{\mathtt{x}}}-{\mathpzc{x}}({\mathtt{x}}{\upepsilon}))\to0$ as ${\upepsilon}\to0$. 
By the coercivity condition \eqref{CoercK}, the convergence in norm ${\mathpzc{x}}^{{\upepsilon}}\to{\mathpzc{x}}$, hence, follows as claimed. 
\qed
\end{proof}

\begin{figure}[htb!]\centering
\includegraphics[width=\textwidth]{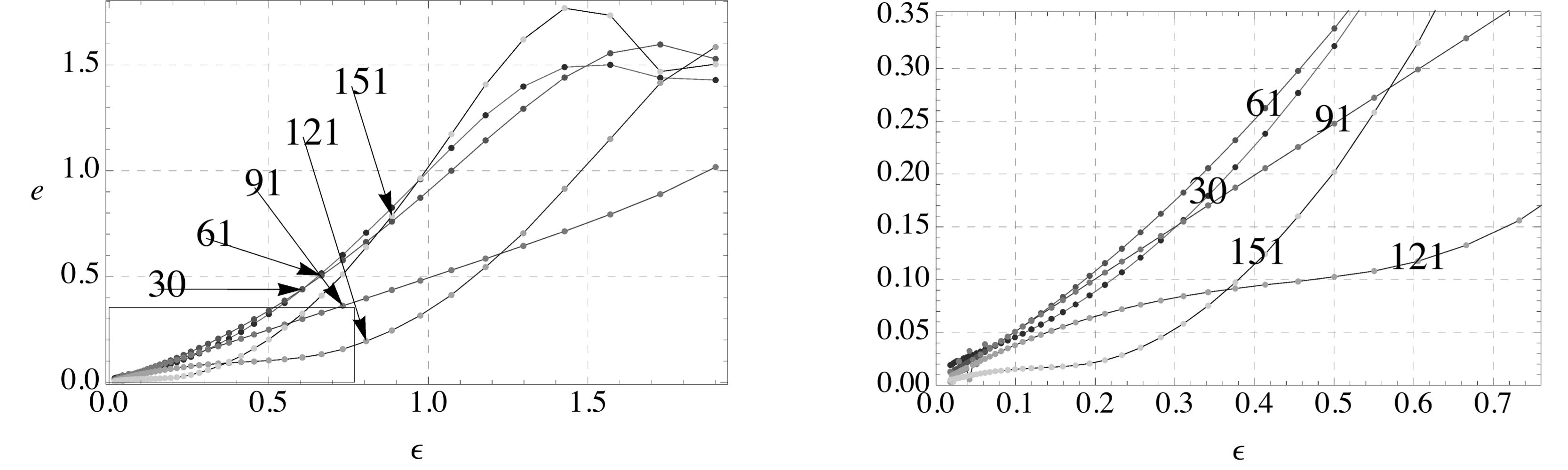}
\caption{\footnotesize $\mathscr{H}^{-{\frac{1}{2}}}$ error vs ${\upepsilon}$ (see \eqref{error}, \eqref{cwx}, \eqref{dWHw}, and \eqref{dSSnorm}) for semi-infinite rigid constraint (`discrete' Dirichlet) with different shades of gray depending on ${{\mathbb{T}}heta}$ as indicated. Right plot is a zoomed portion corresponding to a rectangle shown in the left plot. }\label{HerrorC}\end{figure}

\begin{remark}
For a graphical illustration of the continuum limit for the discrete Dirichlet problem, a numerical evaluation of the relative `error' norm
\begin{eqn}
e({\upepsilon})=\norm{{\mathpzc{x}}^{{\upepsilon}}-\mathcal{R}_{{\upepsilon}}{\mathpzc{x}}}_{-{\frac{1}{2}}}/\norm{\mathcal{R}_{{\upepsilon}}{\mathpzc{x}}}_{\mp{\frac{1}{2}}},
\label{error}
\end{eqn}
in the discrete Sobolev space indicated by the subscript, with respect to ${\upepsilon}$, is considered after a truncation of the infinite sum $\sum_{j\in{\mathbb{Z}}^-}$ by $\sum_{j=-\lfloor X/{\upepsilon}\rfloor}^{-1},$ with $X$ sufficiently large. This is shown in Fig. \ref{HerrorC} 
with a choice $X=76, {\omega}=1$, and using the exact solutions of the discrete Sommerfeld problem and the continuous Sommerfeld problem following the expressions reported by \cite{Bls0} and \cite{Bls2} and \cite{Sommerfeld1, Sommerfeld2, 
ThauPao}, respectively. The dots in Fig. \ref{HerrorC} 
correspond to numerical values for certain choices of ${\upepsilon}$ and are connected by line segments for visual convenience. 
The numerical result provides a graphical illustration of 
Theorem \ref{maineqn}, as 
the `error' $e({\upepsilon})$ goes to zero as ${\upepsilon}\to0$. Analogous illustration of the continuum limit for the discrete Neumann problem has been previously announced and presented by \cite{Bls2} (as Fig. 5). 
\end{remark}
\begin{remark}
For non-zero ${\upepsilon}$, since the discrete Sobolev norms are equivalent to $\ell_2$, graphical results for some large values of ${\upepsilon}$ can be also considered. This topic is discussed by \cite{Bls2} and \cite{Bls3}.
\end{remark}

\section{Conclusion}
In this paper, the continuum limit of discrete Sommerfeld problems, involving a two-dimensional problem of diffraction of a time harmonic lattice wave by the tip of a semi-infinite crack or rigid constraint, on square lattice has been established, 
while utilizing the benefit of hindsight, i.e. the asymptotic results on low frequency approximation stated in recent papers of the author. 
The work of \cite{Hackbusch} provided the required connection with continuous Sommerfeld problems. The results can also be interpreted in terms of error analysis for a numeral method to solve diffraction problems. 
It appears that an alternative proof of the main result is also possible using the results of \cite{EggLubich} and some recent progress in convolution error estimates for non-sectorial kernel \citep{Banjai}, however, the application requires certain estimates of the discretized symbol which are not available for discretization based on the Runge--Kutta methods.
Based on the availability of an extremely large number of rigorous results in numerical analysis of finite difference methods, after presenting the analysis in this paper, a natural question arises: does there exist a shorter proof of the continuum limit? 
However, in the presence of a limited survey, the author of this paper has not come across any such proof.
{Last but not the least, there is also an open problem associated with the limit $\omega_2 \to0+$, as the presented argument breaks down for $\omega_2=0$. The issue involved concerns the analysis of the discrete problem first which is, as yet, an open problem in the discrete Sommerfeld problems as introduced by the author (see \cite{Bls2} just before \S4, and \cite{Bls3} just before \S7.1, for comments concerning the open problem of existence for the limit $\omega_2 \to0+$).}

\section*{Acknowledgements}
This work was partially supported by the grant IITK/ME/20080334 and project 
IITK/ME/20090027. The author thanks A. Anand and P. Munshi for 
suggestions. 
The author thanks an anonymous reviewer of \citep{Bls3} for a comment that prompted the author to prepare a rigorous proof of continuum limit that had been deferred during the preparation of the papers \citep{Bls0}--\citep{Bls3}. 
{The author thanks the anonymous reviewers for several constructive comments and suggestions.}

\appendix
\section{Auxiliary details} 
\label{appL}
The zeros of ${{\mathpzc{H}}}$ are ${{z}}_{{\mathpzc{h}}}$ and $1/{{z}}_{{\mathpzc{h}}}$ while the zeros of ${{\mathpzc{R}}}$ are ${{z}}_{{\mathpzc{r}}}$ and $1/{{z}}_{{\mathpzc{r}}}$, where the complex numbers ${{z}}_{{\mathpzc{h}}}$ and ${{z}}_{{\mathpzc{r}}}$ are given by 
$$
{{z}}_{{\mathpzc{h}}}{:=}{\frac{1}{2}} (2-{\upomega}^2\pm{\upomega}\sqrt{{\upomega}^2-4} ), \text{\rm and }{{z}}_{{\mathpzc{r}}}{:=}{\frac{1}{2}} (6-{\upomega}^2\pm\sqrt{{\upomega}^4-12 {\upomega}^2+32}).
$$
The $\pm$ sign, in front of the square root, is chosen such that ${{z}}_{{\mathpzc{h}}}, {{z}}_{{\mathpzc{r}}}$ lie inside the unit circle ${{\mathbb{T}}}$ (note that ${\upomega}_2>0$). 
The zeros of ${{\mathpzc{Q}}}$ are ${{z}}_{\sq}$ and $1/{{z}}_{\sq}$ with 
$$
{{z}}_{\sq}={\frac{1}{2}} (4-{\upomega}^2-\sqrt{12-8{\upomega}^2+{\upomega}^4})
$$
such that $|{{z}}_{\sq}|<1$. Note that the {\em Wiener--Hopf kernel} ${{\mathpzc{L}}_{{c}}}$ is analytic in ${{\mathscr{A}}}_{{\mathpzc{L}}_{{c}}}$ while ${{\mathpzc{L}}_{{k}}}$ is analytic in ${{\mathscr{A}}}_{{\mathpzc{L}}_{{k}}}$, where
\beqans
{{\mathscr{A}}}_{{\mathpzc{L}}_{{k}}}&{:=}&\{{{z}}\in{\mathbb{C}} : {{\mathit{R}}}_{{\mathpzc{L}}_{{k}}}< |{{z}}|< {{\mathit{R}}}_{{\mathpzc{L}}_{{k}}}^{-1}\}, {{\mathit{R}}}_{{\mathpzc{L}}_{{k}}}{:=}\max\{|{{z}}_{{\mathpzc{h}}}|, |{{z}}_{{\mathpzc{r}}}|\},\label{annALk}\\
{{\mathscr{A}}}_{{\mathpzc{L}}_{{c}}}&{:=}&{{\mathscr{A}}}_{{{\mathpzc{L}}_{{k}}}}\cap\{{{z}}\in{\mathbb{C}} : |{{z}}_{\sq}|< |{{z}}|<|{{z}}_{\sq}|^{-1}\},\label{annALc}
\eeqans{annAc}

\section{Discrete Sobolev spaces}
\label{appS}
Let $\mathscr{H}^s$ ($s\in{\mathbb{R}}$) denote the collection of $u_{{\upepsilon}}: {\upepsilon}{\mathbb{Z}}\to{\mathbb{C}}$ with $\norm{u_{{\upepsilon}}}_s< \infty$, where
\begin{eqn}
\norm{u_{{\upepsilon}}}_s=(\frac{{\upepsilon}}{2\pi})^{{\frac{1}{2}}}\norm{\varpi_{{\upepsilon}}^{s}u_{{\upepsilon}}^F}_{{\mathscr{L}}_2(I)}, I=[-\pi, \pi]\subset {\mathbb{R}},
\label{dSSnorm}
\end{eqn}
and the weight $\varpi_{{\upepsilon}}$ is defined by
\begin{eqn}
\varpi_{{\upepsilon}}({\upxi})=(1+4{\upepsilon}^{-2}\sin^2{\frac{1}{2}}{\upxi})^{{\frac{1}{2}}}, {\upxi}\in I.
\label{dSSweight}
\end{eqn}
In \eqref{dSSnorm}, $u^F_{{\upepsilon}}$ is the discrete Fourier transform \citep{Krein, Silbermann, Slepyanbook} of $u_{{\upepsilon}}$ defined by 
\begin{eqn}
u^F_{{\upepsilon}}({\upxi}) = \sum_{j\in{\mathbb{Z}}}u_{{\upepsilon}}(j{\upepsilon}) e^{-ij{\upxi}}~({\upxi}\in I),
\label{dFT}
\end{eqn}
with inverse transformation given by
\begin{eqn}
u_{{\upepsilon}}(j{\upepsilon}) = \frac{1}{{2\pi}}\int_{I}u^F_{{\upepsilon}}({\upxi}) e^{ij{\upxi}}d{\upxi}.
\label{invdFT}
\end{eqn}
$\mathscr{H}^0$ is the discrete analogue of ${\mathscr{L}}_2({\mathbb{R}})$ and is a Hilbert space with the standard scalar product 
\begin{eqn}
\langle u_{{\upepsilon}}, v_{{\upepsilon}}\rangle = {\upepsilon}\sum_{j\in{\mathbb{Z}}}u_{{\upepsilon}}(j{\upepsilon}) \overline{v_{{\upepsilon}}(j{\upepsilon})},
\end{eqn}
which defines the norm $\norm{u_{{\upepsilon}}}_0=\sqrt{\langle u_{{\upepsilon}}, u_{{\upepsilon}}\rangle}.$ Usually, this Hilbert space is denoted by $\ell_2.$ 
Note that by Plancherel Theorem $\langle u_{{\upepsilon}}, v_{{\upepsilon}}\rangle ={\upepsilon}\langle u_{{\upepsilon}}, v_{{\upepsilon}}\rangle_{{\mathscr{L}}_2(I)}$ and, in particular, 
$\norm{u_{{\upepsilon}}}_0= (\frac{{\upepsilon}}{2\pi})^{{\frac{1}{2}}}\norm{u_{{\upepsilon}}^F}_{{\mathscr{L}}_2(I)}.$
So $\mathscr{H}^0$ consists of all complex-valued functions $u_{{\upepsilon}}$ with finite $\norm{\cdot}_0$ norm defined by
$\norm{u_{{\upepsilon}}}_0= {\upepsilon}^{{\frac{1}{2}}}(\sum_{j\in{\mathbb{Z}}}|u_{{\upepsilon}}(j{\upepsilon})|^2)^{{\frac{1}{2}}}.$

Let 
$$
\langle u_{{\upepsilon}}, v_{{\upepsilon}}\rangle_s={\upepsilon}\int_I \varpi_{{\upepsilon}}^{2s}({\upxi})u_{{\upepsilon}}^F({\upxi})\overline{v_{{\upepsilon}}^F({\upxi})}d{\upxi}, u_{{\upepsilon}}, v_{{\upepsilon}}\in\mathscr{H}^s,
$$
so that $\norm{u_{{\upepsilon}}}_s=\sqrt{\langle u_{{\upepsilon}}, u_{{\upepsilon}}\rangle_s}.$ $\norm{u_{{\upepsilon}}}_s$ is a natural definition since for $s = n= 1$ it coincides with the usual definition
\begin{eqn}
\norm{u_{{\upepsilon}}}^\ast_n =(\sum_{|\alpha|\le n, \alpha\in{\mathbb{Z}}_0^+}|\partial^\alpha_{{\upepsilon}}u_{{\upepsilon}}|_0^2)^{{\frac{1}{2}}}~(n\in{\mathbb{Z}}^+_0),
\label{}
\end{eqn}
where $\partial^\alpha_{{\upepsilon}}={\upepsilon}^{-1}(\id{}_{{\upepsilon}}-T^{-1}_{{\upepsilon}}), T_{{\upepsilon}}u_{{\upepsilon}}(x)=u_{{\upepsilon}}(x+{\upepsilon}).$ For general $n=s\in{\mathbb{Z}}_0^+,$ there exist ${\mathit{C}}_1, {\mathit{C}}_2>0$ such that ${\mathit{C}}_1\norm{\cdot}_n^\ast\le\norm{\cdot}_n\le {\mathit{C}}_2\norm{\cdot}_n^\ast.$ It is easily verified that 
$\norm{u_{{\upepsilon}}}_n$ is equivalent to $\norm{u_{{\upepsilon}}}^\ast_n$ for $n\in{\mathbb{Z}}^+.$ 
For each ${\upepsilon}$, all $s$-norms are equivalent, but, indeed, the equivalence does not hold uniformly in ${\upepsilon}$. Also for $s<t$,
$\norm{\cdot}_t\le c(t-s){\upepsilon}^{s-t}\norm{\cdot}_s.$

Note that for $s\ge 0$, 
$$
\varpi_{{\upepsilon}}^{s}({\upxi})<(1+{\upepsilon}^{-2}{\upxi}^2)^s, {\upxi}\in(-\pi, \pi),
$$ 
while for $s<0$,
$$
\varpi_{{\upepsilon}}^{s}({\upxi})>(\frac{2}{\pi})^s(1+{\upepsilon}^{-2}{\upxi}^2)^s, {\upxi}\in(-\pi, \pi).
$$
Let $u_{{\upepsilon}}=\mathcal{R}_{{\upepsilon}}u, u\in {\mathscr{H}}^s({\mathbb{R}})$. Then, for $s\ge 0$,
\begin{eqn}
\norm{u_{{\upepsilon}}}_s^2&=(\frac{{\upepsilon}}{2\pi})\norm{\varpi_{{\upepsilon}}^{s}u_{{\upepsilon}}^F}^2_{{\mathscr{L}}_2(I)}=(\frac{{\upepsilon}}{2\pi})\int_{-\pi}^{\pi}\varpi_{{\upepsilon}}^{2s}({\upxi}){|u^F_{{\upepsilon}}|}^2({\upxi})d{\upxi}\\&=\frac{1}{2\pi}\int_{-{\upepsilon}^{-1}\pi}^{{\upepsilon}^{-1}\pi}\varpi_{{\upepsilon}}^{2s}({\upepsilon}{\xi}){\upepsilon}^2{|u^F_{{\upepsilon}}|}^2({\upepsilon}{\xi})d{\xi}\\
&\le\frac{1}{2\pi}\int_{-{\upepsilon}^{-1}\pi}^{{\upepsilon}^{-1}\pi}(1+{\xi}^2)^s{\upepsilon}^2{|u^F_{{\upepsilon}}|}^2({\upepsilon}{\xi})d{\xi}\le{\mathit{C}}\norm{u}_{\mathscr{H}^s({\mathbb{R}})}^2,
\label{DCSob1}
\end{eqn}
and, by a similar argument, $\norm{u}_{\mathscr{H}^s({\mathbb{R}})}\le{\mathit{C}}\norm{u_{{\upepsilon}}}_s$. For $s<0$ also, it can be shown that $\norm{u_{{\upepsilon}}}_s^2\le{\mathit{C}}\norm{u}_{\mathscr{H}^s({\mathbb{R}})}^2$ and $\norm{u}_{\mathscr{H}^s({\mathbb{R}})}\le{\mathit{C}}\norm{u_{{\upepsilon}}}_s$.

\end{document}